\documentclass[11pt]{amsart}
\usepackage{amsmath, amssymb, amsthm}
\usepackage{graphicx}
\usepackage{enumitem}
\usepackage{float}

\usepackage[a4paper, margin=1in]{geometry}

\title{The irreducibility and monodromy of some families of linear series with imposed
ramifications}
\author{Xiaoyu Hu}
\date{\today}

\newtheorem{theorem}{Theorem}[section]
\newtheorem{lemma}[theorem]{Lemma}

\newtheorem{proposition}[theorem]{Proposition}

\theoremstyle{definition}
\newtheorem{definition}[theorem]{Definition}
\newtheorem{example}[theorem]{Example}
\newtheorem{remark}[theorem]{Remark}

\begin{document}

\begin{abstract}
Suppose that the adjusted Brill-Noether number is zero,
we prove that there exists a family of twice-marked smooth projective curves such that
the family of linear series with two imposed ramification conditions 
is irreducible. Moreover, under certain  conditions, we show that the monodromy group
contains the alternating group. In the case $r=1$, the monodromy group is the 
full symmetric group.
\end{abstract}

\maketitle

\section{Introduction}

Let $C$ be a smooth projective curve of genus $g$ over 
 $\mathbb{C}$,
 then all the linear series of degree $d$ and rank $r$  on $C$ 
 form a projective moduli space $G^r_d(C)$, which is known as the Brill-Noether locus. 
 The classical Brill-Noether theorem\cite{GH80} states that,
  if $C$ is a general projective curve
  and
the Brill-Noether number 
\begin{equation*}
    \rho(g,r,d)=g-(r+1)(g-d+r)\geq 0,
\end{equation*}
 then $G^r_d(C)$ is a smooth variety of dimension exactly $\rho(g,r,d)$\cite{Gi82}.
 Moreover, if $\rho(g,r,d)>0$, the variety $G^r_d(C)$ is also irreducible.

Eisenbud and Harris\cite{EH87} considered the analogy of 
such irreducibility when $\rho(g,r,d)=0$. If $\rho(g,r,d)=0$, then 
the variety $G^r_d(C)$ is reducible and precisely consists of
\begin{equation*}
    N(g,r,d)=g!\prod_{i=0}^r\frac{i!}{(g-d+r+i)!}
\end{equation*}
reduced points. However, if we consider a family of curves $\mathcal{C}\to B$,
the corresponding family of linear series 
$G^r_d(\mathcal{C}/B)$ may be irreducible. 
Eisenbud and Harris proved the following theorem:
\begin{theorem}[\cite{EH87}, Theorem 1]
There exists a family of smooth projective curves $\mathcal{C}/B$ such that 
$G^r_d(\mathcal{C}/B)$ is irreducible.
\end{theorem}

In this paper, we generalize this theorem to the case of linear series
 with imposed
 ramification conditions at two generic points. Let $\alpha,\beta$ be two ramification sequences. Given
  a smooth projective curve $C$ and two distinct points $p,q$ on $C$, 
 all the linear series of degree $d$ and rank $r$ on $C$ with ramification
  sequence at least $\alpha$ at $p$ and at least $\beta$ at $q$ also
  form a  projective moduli space $G^{r}_d(C,(p,\alpha),(q,\beta))$\cite{EH86}. 
  Eisenbud and Harris\cite{EH86} proved that
if $C$ is a general projective curve,
 then $G^{r}_d(C,(p,\alpha),(q,\beta))$ has dimension exactly
\begin{equation*}
\rho(g,r,d,\alpha,\beta)=g-(r+1)(g-d+r)-|\alpha|-|\beta|,
\end{equation*}
which is often called the adjusted Brill-Noether number.

Now we consider the case of $\rho(g,r,d,\alpha,\beta)=0$. In this case, $G^{r}_d(C,(p,\alpha),(q,\beta))$ consists of precisely 
$N(g,r,d,\alpha,\beta)$ reduced points. We will
 prove the following theorem:

\begin{theorem}[Theorem \ref{irred}]\label{irred2}
    Let $\alpha,\beta$ be two ramification sequences, and suppose
    \begin{equation*}
    \rho(g,r,d,\alpha,\beta)=g-(r+1)(g-d+r)-|\alpha|-|\beta|=0.
    \end{equation*}
    Then there exists a family of twice-marked smooth projective curves
    $(\mathcal{C}/B,p,q)$ such that $G^r_d(\mathcal{C}/B,(p,\alpha),(q,\beta))$
     is
    irreducible.
\end{theorem}
Moreover, we can demonstrate that, in some cases, the monodromy group of the family given in the proof of Theorem \ref{irred2} is at least the alternating group. 
This result also generalizes Edidin's theorem (\cite{Ed91}, Theorem 2).

\begin{theorem}[Theorem \ref{monodromyalter}, Proposition \ref{rank1case}]\label{monodromyalter2}
    Given two ramification sequences $\alpha,\beta$. Suppose
    \begin{equation*}
    \rho(g,r,d,\alpha,\beta)=g-(r+1)(g-d+r)-|\alpha|-|\beta|=0.
    \end{equation*}
     Let $N=N(g,r,d,\alpha,\beta)$.
    If $\alpha,\beta$ satisfy the conditions of
     Lemma \ref{generaltransitive},
    then the monodromy group of 
    $G^r_d(\mathcal{C}/B,(p,\alpha),(q,\beta))$ in
     Theorem \ref{irred2} is either the alternating group $A_N$ or 
     the symmetric group $S_N$. Moreover, if $r=1$, then the monodromy group 
     is $S_N$.
\end{theorem}

Brill-Noether loci with fixed ramification have been studied in some 
literature, such as \cite{CAPB15},\cite{Bi23}.

\subsection*{Notational conventions}
We mention some conventions that will be used throughout the paper.
\begin{enumerate}
  \item The number $g,r,d$ will always be nonnegative and satisfy $g-d+r\geq 0$.
  \item The symbols $a=(a_i)_i$, $b=(b_i)_i$ will refer to vanishing sequences of a linear series. 
  The numbers $a_i$ and $b_i$ will be in increasing order. The symbols 
  $\alpha=(\alpha_i)_i$, $\beta=(\beta_i)_i$ will refer to the corresponding
  ramification sequences of a linear series, defined by $\alpha_i=a_{r-i}-(r-i)$,
  $\beta_i=b_{r-i}-(r-i)$. 
  The numbers $\alpha_i$ and $\beta_i$ will be in nonincreasing order.
\end{enumerate}
\subsection*{Structure of the paper}
 In Section~\ref{sec:prelim}, we proved the  necessary background on limit linear 
 series. Section 3 explores the combinatorics of limit linear series
 and monodromy actions on a certain reducible curve. 
 In Section 4, we prove the irreducibility result and establish
  some technical combinatorial 
 lemmas. Finally, in Section 5, we prove
  our main results about monodromy groups.

\subsection*{Acknowledgements}
I would like to thank Xiang He for suggesting this problem and 
helpful conversations.

\section{Preliminaries}\label{sec:prelim}

Let $C$ be a smooth projective curve over $\mathbb{C}$.
\begin{definition}
    A \textit{linear series} on a curve $C$ of degree $d$ and rank $r$ is a
     pair $(\mathcal{L}, V)$, 
    where $\mathcal{L}$ is a line bundle of degree $d$ on $C$,
     and $V$ is an $(r+1)$-dimensional 
    subspace of $\mathrm{H}^0(C,\mathcal{L})$.
  \end{definition}
  A linear series of degree $d$ and rank $r$ is historically
   denoted as $\mathfrak{g}^r_d$.
  
  For any linear series, we can define 
  its vanishing sequence and ramification sequence at each point.
  \begin{definition}
    Let $L=(\mathcal{L},V)$ be a  $\mathfrak{g}^r_d$ on a curve $C$. For any 
    $p\in C$, the \textit{vanishing sequence} of $(\mathcal{L},V)$ at $p$ is the 
    increasing sequence 
    of orders of sections in $V$ at the point $p$:
    \begin{equation*}
      0 \leq a_0(p)<a_1(p)<\cdots<a_r(p)\leq d
    \end{equation*}
    and the corresponding \textit{ramification sequence} is the non-increasing sequence
    \begin{equation*}
      d-r\geq \alpha_0(p)=a_r(p)-r\geq \alpha_1(p)=a_{r-1}(p)-(r-1)\geq\cdots\geq 
      \alpha_r(p)=a_0(p)-0\geq 0,
    \end{equation*}
    where $\alpha_i(p)=a_{r-i}(p)-(r-i)$.
  
    If $\alpha_i(p)=0$ for all $0\leq i\leq r$, then $p$ is called 
    a \textit{non-ramification point}; otherwise, 
    it is a \textit{ramification point}.
  \end{definition}
  Some literature defines the ramification sequence as a 
  non-decreasing sequence, i.e., $\alpha_i(p)=a_{i}(p)-i$. 
  Here, to match with Schubert indices later on, 
  we define it as a non-increasing sequence. We say a linear series has 
  ramification at least $\alpha$ at $p$ if its ramification sequence at $p$
  is greater than or equal to $\alpha$ term by term.

The following theorem is a generalization of the Brill-Noether theorem, proved 
by Eisenbud and Harris.
\begin{theorem}[\cite{EH86}, Theorem 4.5]
  Given $g,r,d,n$ and ramification sequences $\alpha^j$ for each $j=1,\cdots,n$, let  
  \begin{equation*}
    \rho=g-(r+1)(g-d+r)-\sum_{j=1}^n|\alpha^j|.
  \end{equation*}
  Then for all smooth projective curves $C$ of genus $g$, and distinct 
  marked points $p_1,\cdots,p_n\in C$, there is a projective moduli space 
  $G^r_d(C,(p_1,\alpha^1),\cdots,(p_n,\alpha^n))$ of $\mathfrak{g}^r_d$'s on $C$ 
  with ramification at least $\alpha^j$ at each $p_j$, and it has every 
  component of dimension at least $\rho$ if it is non-empty. On a general curve $C$
  of genus $g$, the space of $\mathfrak{g}^r_d$'s has dimension exactly $\rho$, and 
  in particular is empty if $\rho<0$.
\end{theorem}

\subsection{Limit linear series}
The original proofs of the Brill-Noether theorem \cite{GH80},
 along with many other results, 
 relied on analyzing linear series on smooth curves 
 and studying their degeneration to singular curves. 
 Later, Eisenbud and Harris \cite{EH86} developed the theory of
  limit linear series, which describes how linear series
   specialize as smooth curves degenerate into curves of 
   compact type. This theory, known for its ability to characterize the
    behavior of linear series under degenerations, 
   has become a powerful tool in algebraic geometry.

  A (possibly) reducible curve $C$ is of \textit{compact type} if
  its Jacobian is compact, or equivalently, if its
   irreducible components are  smooth and meet transversely at a time
    and its dual graph obtained by considering the components
     as vertices and intersection relationships as edges is a tree.

\begin{definition}
Let $C$ be a curve of compact type. For each irreducible component $Y$ of $C$,
 let $L_Y=(\mathcal{L}_Y,V_Y)$ be a $\mathfrak{g}^r_d$ on $Y$. If the collection
\begin{equation*}
L = \{L_Y=(\mathcal{L}_Y,V_Y): Y\text{ is an irreducible component of $C$}\}
\end{equation*}
satisfies, for any two irreducible components $Y$ and $Z$ of $C$ with 
$p=Y\cap Z$, and for any $j=0,\cdots,r$, we have 
\begin{equation*}
a_j^{L_Y}(p)+a^{L_Z}_{r-j}(p)\geq d,
\end{equation*}
then $L$ is called a \textit{crude limit linear series of degree $d$ and rank $r$} on $C$. 
If all these inequalities are equalities, then $L$ is called a
 \textit{refined limit linear series}, or 
 simply a \textit{limit linear series}, on $C$.
\end{definition}
If $L$ is a crude limit linear series but not refined, we call $L$ a \textit{strictly}
crude limit linear series.

Eisenbud and Harris constructed a scheme parameterizing linear series and 
limit linear series for any smoothing family\cite{EH86}, and conclude the following:
\begin{theorem}[\cite{EH86}, Theorem 3.3]
  Let $\mathcal{C}/B, p_1,\cdots,p_n$ be an 
  ($n$-pointed genus-$g$ curve) smoothing family 
  and $\alpha^1,\cdots,\alpha^n$ be ramification sequences. Let 
  \begin{equation*}
    \rho = g-(r+1)(g-d+r)-\sum_{j=1}^n|\alpha^j|.
  \end{equation*}
  Then there is a quasiprojective scheme $G=G^r_d(\mathcal{C}/B,(p_1,
  \alpha^1),\cdots,(p_n,\alpha^n))$ parameterizing linear series on smooth fibers 
  of $\mathcal{C}$, and limit linear series on singular fibers of $\mathcal{C}$, 
  both of degree $d$ and dimension $r$, and having ramification at least 
  $\alpha^j$ at each $p_j$ for $j=1,\cdots,n$. 
  The dimension of any component of 
  $G$ is at least $\rho+\dim(B)$.

  If either 
  \begin{equation*}
    \sum_{j=1}^n|\alpha^j|=(r+1)d+\binom{r+1}{2}(2g-2),
  \end{equation*}
  or no reducible fibers $C_0$ of $\mathcal{C}$ have strictly crude $\mathfrak{g}^r_d$'s 
  with the prescribed ramifications, then $G$ is proper over $B$.
\end{theorem}

\subsection{Linear series on $\mathbb{P}^1$ and Schubert cycles}
 It is well known that
 linear series of
degree $d$ and rank $r$ on $Y=\mathbb{P}^1$  correspond to $(r+1)$-dimensional
 subspaces of $\mathrm{H}^0(Y,\mathcal{O}_Y(d))$. 
 Let $\mathrm{Gr}(r,d)$ be the Grassmannian parameterizing all 
 $(r+1)$-dimensional subspaces of 
 $\mathrm{H}^0(Y,\mathcal{O}_Y(d))$, or equivalently,
  all $r$-dimensional subspaces of $\mathbb{P}^d$.

  Ramification sequences of linear series on $Y=\mathbb{P}^1$
are closely connected to Schubert indices.
For any point $q\in Y$,  we define subspaces $f^i(q)$ 
  \begin{equation*}
  f^i(q)=\{\sigma\in\mathrm{H}^0(Y,\mathcal{O}(d)): 
  \mathrm{ord}q(\sigma)\geq d-i\}.
  \end{equation*}
  These subspaces
  form a complete flag in $\mathrm{H}^0(Y,\mathcal{O}(d))$. For a
   Schubert 
  index $\alpha=(\alpha_0,\cdots,\alpha_r)$ satisfying
  \begin{equation*}
  d-r\geq\alpha_0\geq\cdots\geq\alpha_r\geq 0,
  \end{equation*}
  we define the associated Schubert cycle in the Grassmannian as
  \begin{equation*}
  \sigma_{\alpha}(q)=\{V\in \mathrm{Gr}(r,d): \mathrm{dim}(V\cap f^{(d-r+i-\alpha_i)}(q))>i\},
  \end{equation*}
  which is a subvariety of codimension $|\alpha|=\sum_i\alpha_i$. 
  The relationship between ramification sequences and Schubert varieties
   can be expressed as follows:
  \begin{proposition}
    If $(\mathcal{O}_Y(d),V)$ is a linear series on $Y=\mathbb{P}^1$, 
    then its ramification sequence at point $q$ is greater
     than or equal to $(\alpha_0,\cdots,\alpha_r)$ term by term
      if and only if $V\in \sigma_{\alpha_0,\cdots,\alpha_r}(q)$.
    \end{proposition}
    \begin{proof}
    The proof is straightforward. 
    Suppose the vanishing sequence of $(\mathcal{O}_Y(d),V)$ at $q$  is 
    $(b_0,\cdots,b_r)$ and the corresponding ramification sequence is
    $(\beta_0,\cdots,\beta_r)$. On the other hand, let $(a_0,\cdots,a_r)$
     be the corresponding vanishing sequence of $(\alpha_0,\cdots,\alpha_r)$. 
     Then we have 
    \begin{equation*}
    \begin{aligned}
    V\in \sigma_{\alpha_0,\cdots,\alpha_r}(q)
    &\Leftrightarrow \text{ for any }i,\:
     \mathrm{dim}(V\cap f^{d-r+i-(a_{r-i}-(r-i))})>i,\\
    &\Leftrightarrow \text{ for any }i,\: 
    \mathrm{dim}(V\cap f^{d-a_{r-i}})>i,\\
    &\Leftrightarrow \text{ for any }i,\: \mathrm{dim}\{\sigma\in V: \mathrm{ord}_q(\sigma)\geq
     a_{r-i}
    \}\geq i+1,\\
    &\Leftrightarrow \text{ for any }i,\: \mathrm{codim}\{\sigma\in V: \mathrm{ord}_q(\sigma)\geq a_{r-i}
    \}\leq r-i,\\
    &\Leftrightarrow \text{ for any }i,\: b_{r-i}\geq a_{r-i},\\
    &\Leftrightarrow \text{ for any }i,\: \beta_i\geq \alpha_i.\\
    \end{aligned}
    \end{equation*}
    \end{proof}

    Let $q_1,\cdots,q_k$ denote distinct points of $\mathbb{P}^1$ and
     $\alpha^{(1)},\cdots,
    \alpha^{(k)}$  represent ramification sequences associated with these points. 
    An important fact that we will frequently use, which can be found in \cite{EH83}, 
    is that 
    the intersection $\bigcap_{i=1}^k\sigma_{\alpha^{(i)}}(q_i)$ has the expected
    codimension $\sum_{i=1}^k |\alpha^{(i)}|$.

    \vspace{\baselineskip}

    \section{Limit linear series with two imposed ramifications}

     As in \cite{EH87}, we prove our results by
     degenerating smooth curves into reducible nodal curves, leveraging
      the combinatorial representations of limit linear series on
       such curves. A commonly used degeneration method involves chains 
       of elliptic and rational curves.
    Specifically, we examine a twice-marked reducible nodal curve $C_\infty$, 
    which is same as the one discussed in \cite{EH87}, 
    as illustrated in Figure \ref{cinfty}.
    
    \begin{figure}[H]
    \centering
    \includegraphics[width=0.4\textwidth]{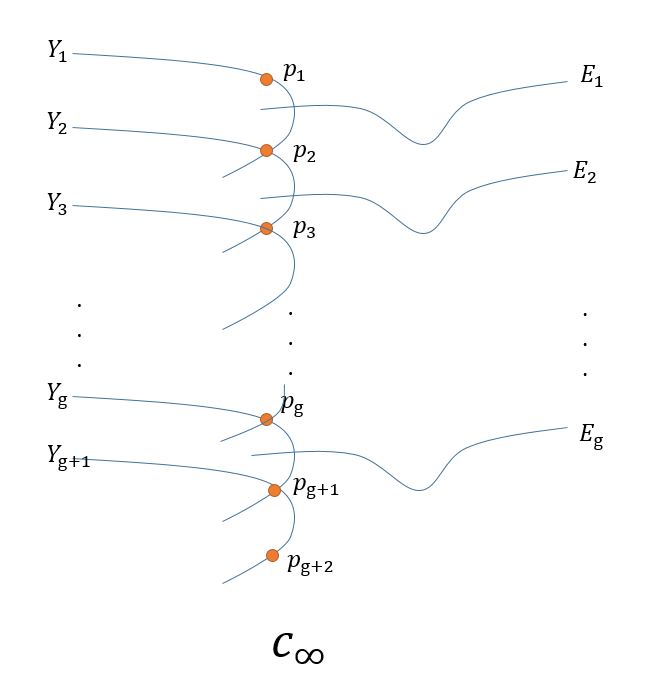}
    \caption{Curve $C_\infty$}
    \label{cinfty}
    \end{figure}
    
    Here, $Y_1,\cdots,Y_{g+1}$ are rational curves isomorphic to $\mathbb{P}^1$, 
    $E_1,\cdots,E_g$ are elliptic curves attached
      to $Y_1,\cdots,Y_g$ respectively, and $p_1=p$, $p_{g+2}=q$ are the
      marked points. To ensure the stability of $C_\infty$, we may label 
      an additional point on $Y_{g+1}$. Although for simplicity of notation, 
      we temporarily omit it. Later, we will 
      provide a detailed description of all limit linear series
       on $C_\infty$ with ramification sequences at least $\alpha$ at $p_1$ and 
       at least $\beta$ at $p_{g+2}$.
    
    To investigate the monodromy action on
     the limit linear series of the curve $C_\infty$, 
     we can examine  some one-parameter families where each curve also takes
      the form of chains of elliptic and rational curves. 
    For any $1\leq i\leq g$, we consider the one-parameter family
     $C_{i,p}$, where as $p$ approaches infinity, 
     the stable limit of $C_{i,p}$ is $C_\infty$. 
     The curve $C_{i,p}$ is depicted in the Figure \ref{cip}.
    
    \begin{figure}[H]
    \centering
    \includegraphics[width=0.4\textwidth]{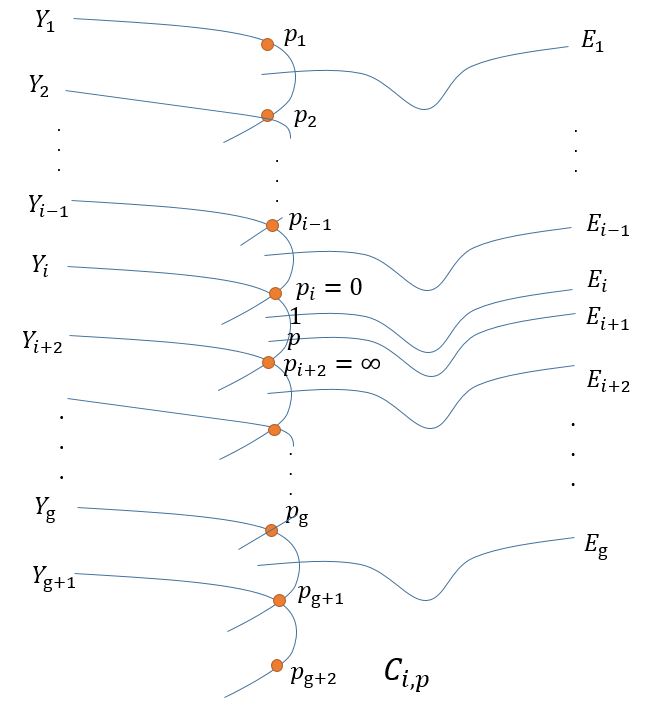}
    \caption{Curve $C_{i,p}$}
    \label{cip}
    \end{figure}
    
    The curve $C_{i,p}$ differs from $C_\infty$ 
    in that it excludes the rational curve  $Y_{i+1}$
    and instead attaches the elliptic curve $E_{i+1}$ to $Y_i$. 
    In terms of coordinates on $Y_i$, 
    the intersection points of  $Y_{i-1},E_i,E_{i+1}$ and $Y_{i+1}$ with
     $Y_i$ are 
 $0,1,p,\infty,$ respectively, where $p\neq 0,1,\infty$.
  As $p$ tends to infinity, the stable limit of the curves $C_{i,p}$ 
  results in the blow-up of the curve at $p=\infty$, 
  precisely adding another rational curve  $\mathbb{P}^1$, which is 
     exactly the curve $C_\infty$.
    
    \subsection{Combinatorial representation of limit linear series on $C_\infty$ 
    and  $C_{i,p}$}

    Next, we study the combinatorial representation
    of limit linear series with ramification conditions at the given two points
     on 
    the curves $C_{i,p}$ and $C_\infty$. We will utilize
     some results from \cite{EH87}, and for convenience, we include relevant propositions here.
    The curves $C_{i,p}$ and $C_\infty$  primarily have only two distinct local types, 
    denoted as $D$ and $D'$, as shown in 
    Figure \ref{curveD}.
    
    \begin{figure}[H]
    \centering
    \includegraphics[width=0.4\textwidth]{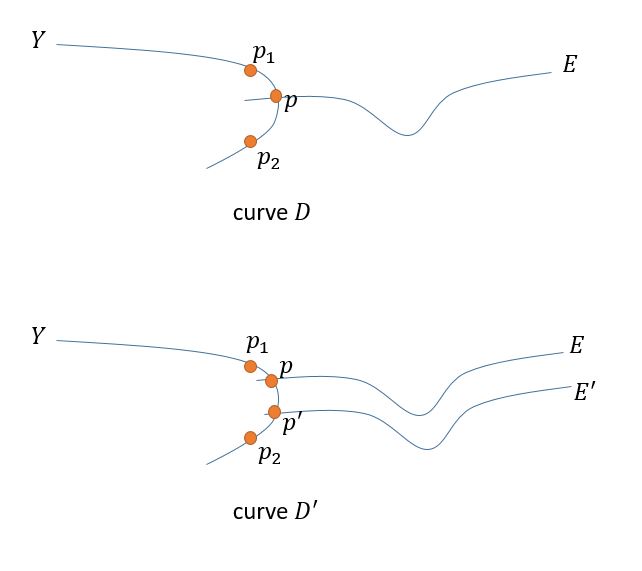}
    \caption{Curves $D$ and $D'$}
    \label{curveD}
    \end{figure}
    
    Let $L$ be a limit $\mathfrak{g}^r_d$ on the curve $D$. Then, $L$ consists of 
      linear series on the two irreducible 
    components of $D$: $L_Y=(\mathcal{O}_Y(d),V_Y)$ 
    and $L_E=(\mathcal{L}_E,V_E)$. 
    We denote the vanishing sequence of $L_Y$ at $p_1$ as $a^{(1)}$
    and at $p_2$ as $b^{(1)}$,  corresponding to ramification
     sequences $\alpha^{(1)}$ and $\beta^{(1)}$ respectively. 
     Furthermore, let
     $\alpha^{(2)}=(\beta^{(1)})^\vee$ denote the dual ramification sequence of $\beta^{(1)}$, where for any $0\leq i\leq r$,
    \begin{equation*}
    \alpha^{(2)}_i=d-r-\beta^{(1)}_{r-i}.
    \end{equation*}
    For the curve $D'$, we use the same notation.  Then the limit linear 
    series on $D$ and $D'$ can be characterized as follows:

    \begin{lemma}[\cite{EH87}, Corollary 1.2]\label{curvegrd}
        We have:
        \begin{enumerate}
        \item On the curve $D$, it holds that $|\alpha^{(2)}|\geq |\alpha^{(1)}|+r$. If equality holds,
        then
        $\mathcal{L}_E=\mathcal{O}_E(dp)$, $V_E$ is the image of
         $\mathrm{H}^0(\mathcal{L}_E(-(r+1)p))$
        in $\mathrm{H}^0(\mathcal{L}_E)$, and there exists a unique $i$ such that
        \begin{equation*}
        b_{r-i}^{(1)}=d-a_i^{(1)},
        \end{equation*}
        and for any $j\neq i$,
        \begin{equation*}
        b_{r-j}^{(1)}=d-a_j^{(1)}-1.
        \end{equation*}
        Conversely, given vanishing sequences $a^{(1)}$ and $b^{(1)}$ satisfying
         these conditions, there exists a unique limit
          $\mathfrak{g}^r_d$ on $D$ such that the vanishing sequences 
          at $p_1$ and $p_2$ are respectively $a^{(1)}$ and $b^{(1)}$.
        
        \item On the curve $D'$, it holds that $|\alpha^{(2)}|\geq |\alpha^{(1)}|+2r$. If equality holds,
        then the linear series $(\mathcal{L}_E,V_E)$ and
         $(\mathcal{L}_{E'},V_{E'})$ of the elliptic curve parts 
          are determined in the same way as $D$.
        
        Given vanishing sequences $a^{(1)}$ and $b^{(1)}$ such that 
        the corresponding ramification sequences satisfy 
        $|\alpha^{(2)}|= |\alpha^{(1)}|+2r$,  there is at most one limit $\mathfrak{g}^r_d$
       on $D'$ with these sequences except in the following case, 
       where there are either $1$ or $2$ such series: 
      there exist integers $i<j$ such that
        \begin{enumerate}
          \item If $i>0$, then $a^{(1)}_{i-1}<a_i^{(1)}-1$;
          \item $a^{(1)}_{j-1}<a^{(1)}_j-1$;
          \item $b^{(1)}_{r-i}=d-a^{(1)}_i-1$;
          \item $b^{(1)}_{r-j}=d-a^{(1)}_j-1$;
          \item For any $k\neq i,j$, $b^{(1)}_{r-k}=d-a^{(1)}_k-2$.
        \end{enumerate}
        \end{enumerate}
        \end{lemma}
        \vspace{\baselineskip}

    Using this lemma, We can fully characterize limit linear series
     on $C_{i,p}$ and $C_\infty$ with the imposed ramification conditions
      $\alpha$ and $\beta$ at $p_1$ and $p_{g+2}$ respectively. 
      Let $L$ be such a limit $\mathfrak{g}_d^r$ on $C_\infty$, then the
       restriction of $L$  on each part  
   remains a limit
       $\mathfrak{g}_d^r$. 
       Suppose   the vanishing
        sequence  and the ramification sequence of $(\mathcal{L}_{Y_i},V_{Y_i})$ at $p_i$ are
         $a^{(i)}=(a^{(i)}_0,\cdots,a_r^{(i)})$ and  $\alpha^{(i)}=(a_r^{(i)}-r,\cdots,a_0^{(i)}-0)$, respectively.
    According to Lemma \ref{curvegrd}, for any $1\leq i\leq g$, we have
    \begin{equation*}
    |\alpha^{(i+1)}|\geq |\alpha^{(i)}|+r,
    \end{equation*}
    which leads to
    \begin{equation}
    |\alpha^{(g+1)}|\geq |\alpha^{(1)}|+rg\geq |\alpha|+rg.
    \label{geqram2}
    \end{equation}
    On the other hand, 
    to ensure the intersection
     $\sigma_{\alpha^{(g+1)}}\cdot\sigma_{\beta}$ is nonempty, 
     it is required that
    \begin{equation*}
    |\alpha^{(g+1)}|+|\beta|\leq\mathrm{dim}\, G(r,d)=(r+1)(d-r).
    \end{equation*}
    While at the beginning we assume 
    \begin{equation*}
    \rho(g,r,d,\alpha,\beta)=g-(r+1)(g-d+r)-|\alpha|-|\beta|=0,
    \end{equation*}
    so that  
    \begin{equation}
    \begin{aligned}
    |\alpha^{(g+1)}|+|\beta| &\leq (r+1)(d-r)\\
    &= rg + |\alpha| + |\beta|.\\
    \end{aligned}
    \label{leqdim}
    \end{equation}
    By combining inequalities \eqref{geqram2} and \eqref{leqdim}, we deduce
     that all the 
    equalities must hold. Thus, $\alpha^{(1)}=\alpha$ and 
    for any $1\leq i\leq g$, $|\alpha^{(i+1)}|=|\alpha^{(i)}|+r$.

    For the curve $C_{i,p}$, we use the same notation. 
    For any $1\leq j\leq g$ and $j\neq i,i+1$, we have
    \begin{equation*}
    |\alpha^{(j+1)}|\geq |\alpha^{(j)}|+r,
    \end{equation*}
    and 
    \begin{equation*}
    |\alpha^{(i+2)}|\geq |\alpha^{(i)}|+2r.
    \end{equation*}
    Consequently,
    \begin{equation*}
    |\alpha^{(g+1)}|\geq |\alpha^{(1)}|+(r-2)g+2g\geq |\alpha|+rg.
    \label{geqram}
    \end{equation*}
    Similarly, in this case, the inequality
     \eqref{leqdim} also holds, 
     confirming that these inequalities are all equalities. 
    Therefore, $\alpha^{(1)}=\alpha$ and for any 
    $1\leq j\leq g$ where $j\neq i,i+1$,
     it must be that  $|\alpha^{(j+1)}|=|\alpha^{(j)}|+r$ and 
     $|\alpha^{(i+2)}|=|\alpha^{(i)}|+2r$.

    In summary, a limit linear series satisfying ramification at least $\alpha$ 
    at $p_1$ and at least $\beta$ at $p_2$
     $L$ on the curve $C=C_\infty$ corresponds to a chain of Schubert cycles:
    \begin{equation}
    \Delta(L)=(\sigma_{\alpha^{(1)}}=
    \sigma_\alpha,\: \cdots\: ,\sigma_{\alpha^{(g+1)}}=\sigma_{\beta}^{\vee})
    \label{schubertchain}
    \end{equation}
    where $\sigma_{\beta}^{\vee}$ is the Poincaré dual of $\sigma_{\beta}$.
    This correspondence is indeed one-to-one, which we formulate  as the 
    following theorem:
    
    \begin{theorem}\label{chain}
    Suppose $\alpha,\beta$ are two ramification sequences and
    \begin{equation*}
    \rho(g,r,d,\alpha,\beta)=g-(r+1)(g-d+r)-|\alpha|-|\beta|=0.
    \end{equation*}
    Let $\Sigma$ be the set of sequences of $g+1$ Schubert cycles in $\mathrm{Gr}(r,d)$:
    \begin{equation*}
      s_1=\sigma_\alpha,\cdots,s_{g+1}=\sigma_\beta^{\vee}
      \end{equation*}
      such that for any $1\leq i <g+1$:
      \begin{equation*}
      s_i\cdot \sigma_{1,\cdots,1,0}\cdot s_{i+1}^{\vee}\neq 0.
    \end{equation*}
    Then, the limit linear series $\mathfrak{g}^r_d$ on
     the curve $C_\infty$ satisfying the ramification sequences 
     at $p_1$, $p_{g+2}$ being at least $\alpha$, $\beta$ respectively, 
     correspond bijectively to $\Sigma$, where 
     the limit linear series $L$ corresponds 
     to $\Delta(L)$ as given in Equation \eqref{schubertchain}.
    \end{theorem}
    
    \begin{proof}
      As discussed above, the limit linear series $L$ 
    corresponds to a chain of Schubert cycles $\Delta(L)$.
     For any $1\leq i\leq g$, let the restriction of $L$  
     on the rational curve $Y_i$ be $L_i$. The ramification sequences 
     of $L_i$ at $p_i, \, p_{i+1}$ and its intersection with $E_i$ are
      $\alpha^{(i)},\,(\alpha^{(i+1)})^\vee,\,(1,\cdots,1,0)$ respectively, 
      where $(\alpha^{(i+1)})^\vee$ is the Schubert index of 
      $(\sigma_{\alpha^{(i+1)}})^\vee$. 
      The existence of such $L_i$  implies
    \begin{equation*}
    \sigma_{\alpha^{(i)}}\cdot\sigma_{1,\cdots,1,0}\cdot
    (\sigma_{\alpha^{(i+1)}})^\vee\neq 0,
    \end{equation*}
    hence the sequence
    \begin{equation*}
    \sigma_{\alpha^{(1)}}=\sigma_\alpha,\cdots,\sigma_{\alpha^{(g+1)}}=\sigma_\beta^\vee
    \end{equation*}
    belongs to $\Sigma$.
    
    Conversely, given a sequence in $\Sigma$:
    \begin{equation*}
    s_1=\sigma_\alpha,
    \: \cdots\: ,s_{g+1}=\sigma_\beta^{\vee}
    \end{equation*}
    such that  for any $1\leq i <g+1$,
    \begin{equation*}
    s_i\cdot \sigma_{1,\cdots,1,0}\cdot s_{i+1}^{\vee}\neq 0,
    \end{equation*}
    then by Lemma \ref{curvegrd} (1), 
    there exists a unique limit linear series $L$ on 
    $C_\infty$ such that its ramification sequences at points 
     $p_1$, $p_2$ are $s_1$, $s_2$ respectively.
    \end{proof}
    
    \vspace{\baselineskip}
    
    Combinatorics in Schubert calculus is often represented by
     Young tableaux. Here, we can also transform such sequences 
     of Schubert cycles into Young tableaux, making it easier to handle
      the combinatorics. This approach is particularly useful for studying monodromy actions.
       Bercov and Proctor\cite{BP87} were the first to use Young tableaux to 
       parameterize limit linear series.
    Now, we restate Theorem \ref{chain} using Young tableaux.
    
    Given  ramification sequences $\alpha$ and $\beta$, 
    we construct a (skew) Young diagram $\sigma = \sigma(g,r,d,\alpha,\beta)$ as 
    shown in Figure \ref{youngt}.
    We place the sequence $\alpha$ from bottom to top on the left side 
    of an $(r+1) \times (g-d+r)$ rectangle, 
    while the  sequence $\beta$ from top to bottom on the right side.
    The reversal of $\beta$ is due to the fact that the last term
     in the  sequence of Schubert cycles as above 
    is the Poincaré dual of $\sigma_\beta$.
    The assumption $\rho(g,r,d,\alpha,\beta)=0$ implies that 
     the entire Young diagram consists of $g$ boxes.

     \begin{figure}[H]
      \centering
      \includegraphics[width=0.9\textwidth]{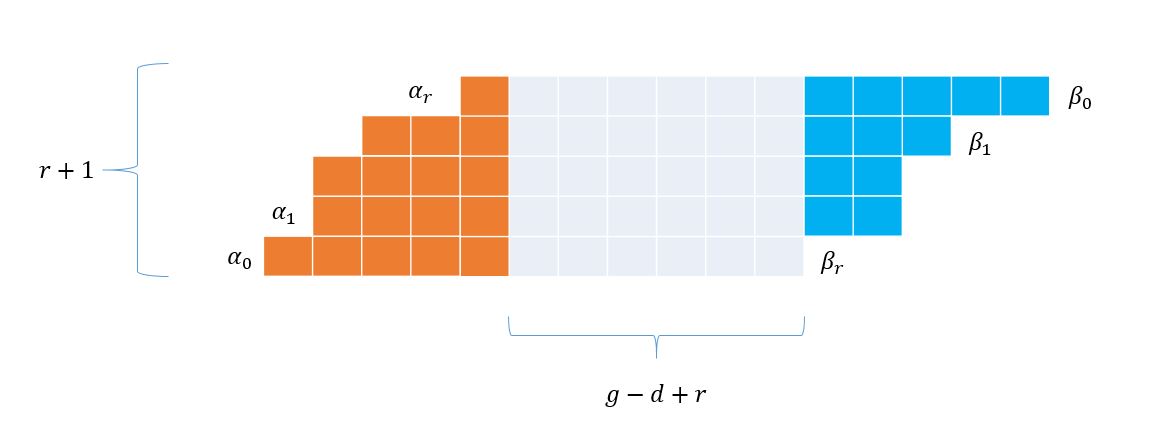}
      \caption{Young diagram corresponding to limit linear series on $C_\infty$}
      \label{youngt}
      \end{figure}
      
    We have 
    \begin{theorem}\label{ytcorrsp}
    Let $\alpha$ and $\beta$ be two ramification sequences, and suppose
    \begin{equation*}
    \rho(g,r,d,\alpha,\beta)=g-(r+1)(g-d+r)-|\alpha|-|\beta|=0.
    \end{equation*}
    Then the limit linear series $\mathfrak{g}^r_d$ on the 
    curve $C_\infty$ with ramification sequences at $p_1$ and $p_{g+2}$ 
    at least $\alpha$ and $\beta$ correspond 
    bijectively to the Young tableaux of the Young diagram $\sigma(g,r,d,\alpha,\beta)$.
    \end{theorem}
    
    \begin{proof}
    By Theorem \ref{chain},  limit linear series on $C_\infty$ correspond
     bijectively to sequences in $\Sigma$. 
     Let $(\sigma_{\alpha^{(0)}}=\sigma_\alpha,
     \cdots,\sigma_{\alpha^{(g+1)}}=\sigma_\beta^{\vee})$ be a sequence
      in $\Sigma$. We now construct 
     the corresponding Young tableau inductively.
    
    Suppose $k\geq 1$ and the integers $1,2,\cdots,k-1$ have already been
     filled in the Young diagram $\sigma(g,r,d,\alpha,\beta)$. According to 
     Lemma \ref{curvegrd}, there exists a unique 
     $0\leq j\leq r$ such that $\alpha^{(k)}_j=\alpha^{(k+1)}_j$. To proceed,
      we fill $k$ into
      the first available empty box from the left in the $(r-j+1)$-th row. 
      The existence of this empty box is 
       guaranteed by the definition of $\Sigma$.
    
       We now demonstrate that the resulting tableau is a Young tableau, 
       meaning it increases from left to right and from top to bottom.  Firstly,
       each row is  increasing from left to right due to our construction. 
     Suppose that there exists a column where adjacent boxes in rows $j$ and 
     $j+1$ are filled with integers $k$ and $k'$ respectively, and $k>k'$. 
     When placing $k'$, the box directly above it must still be empty. 
     This implies $\alpha^{(k')}_{j+1}<\alpha^{(k')}_j$,
      contradicting the nonincreasing property of ramification sequences. 
      Therefore, the constructed tableau is indeed a Young tableau.
    
    Conversely, a Young tableau can be used to construct the 
    original sequence in a similar manner. Thus, limit linear series
     on the curve $C_\infty$ with ramification 
     sequences at $p_1$ and $p_{g+2}$ at least $\alpha$ and $\beta$ 
     correspond bijectively  to the Young tableaux of the Young diagram $\sigma$.
    \end{proof}
    \begin{example}\label{ex:youngtab}
      Consider a curve $C$ of genus $7$ and a $\mathfrak{g}^1_6$ on $C$, denoted 
      as $L$. Let  
      $\alpha = (1,0), \beta = (2,0)$ be ramification sequences. In this case, 
      we have 
      \[\rho(g,r,d,\alpha,\beta) = 7 - (1+1)*(7-6+1)-1-2 = 0.
      \]
      Assume that $L$ is represented as the sequence
      \[\sigma_{\alpha^{(1)}}=\sigma_{1,0},\,
      \sigma_{\alpha^{(2)}}=\sigma_{2,0},\,
      \sigma_{\alpha^{(3)}}=\sigma_{2,1},\,
      \sigma_{\alpha^{(4)}}=\sigma_{3,1},\,
      \sigma_{\alpha^{(5)}}=\sigma_{3,2},
      \]
      \[
      \sigma_{\alpha^{(6)}}=\sigma_{3,3},\,
      \sigma_{\alpha^{(7)}}=\sigma_{4,3},\,
      \sigma_{\alpha^{(8)}}=\sigma_{5,3}=(\sigma_{2,0})^\vee
      \]
      Given that $\alpha^{(1)}$ and $\alpha^{(2)}$   are identical  exactly at index $j=1$: 
      $\alpha^{(1)}_1=\alpha^{(2)}_1$, we place $1$ into the first empty box 
      from the left in the 
      $r-j+1=1$-th row. Similarly, since $\alpha^{(2)}$ and $\alpha^{(3)}$
      are identical exactly  at index $j=0$: 
      $\alpha^{(2)}_0=\alpha^{(3)}_0$, we place $2$ into the first empty box 
      from the left in the 
      $r-j+1=2$-th row. This process continues until all $g=7$ entries are filled, 
      resulting in the corresponding Young tableau shown below:
      \begin{figure}[H]
        \centering
        \includegraphics[width=0.5\textwidth]{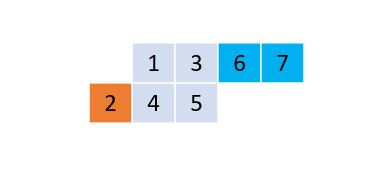}
        \caption{Young tableaux of $L$}
        \label{ytexample}
      \end{figure}
      Conversely, starting from such a Young tableau, we can
       reconstruct the original sequence of Schubert cycles in the same 
        manner.

    \end{example}
    
    \vspace{\baselineskip}

    \subsection{Combinatorial representation of certain monodromy actions}
    
    Next, we investigate the monodromy group action induced by the one-parameter
     family $C_{i,p}$ on the limit linear
     series of the curve $C_\infty$  with imposed ramification conditions.
     To proceed, we need the following lemma:

      \begin{lemma}[\cite{EH87}, Theorem 1.3]\label{branch}
        Let $a,b$ be two vanishing sequences, such that there 
        exist $0\leq i<j\leq r$ such that 
        \begin{enumerate}
          \item if $i>0$,
        $a^{(1)}_{i-1}<a_i^{(1)}-1;$
          \item $a^{(1)}_{j-1}<a^{(1)}_{j-1};$
          \item $b^{(1)}_{r-i}=d-a^{(1)}_{i-1};$
          \item $b^{(1)}_{r-j}=d-a^{(1)}_{j-1};$
          \item $\text{For any }k\neq i,j, \text{ we have }
          b^{(1)}_{r-k}=d-a^{(1)}_k-2,$
        \end{enumerate}
        
        then on $\mathbb{P}^1$, all vanishing sequences at $0$ are
         at least $a$, at $\infty$ are at least $b$,
        and at $1$ and another point $p\neq 0,1,\infty$ are cusps,
         the moduli space of linear series $\mathfrak{g}^r_d$
        is an irreducible rational curve $G$. The map from each
         linear series to the corresponding point $p$ gives a map $G\to\mathbb{P}^1$
        which is a double cover, branching at two points on 
        $\mathbb{P}^1\setminus\{0,1,\infty\}$, and these two
        branching points are determined by $a_j-a_i$. Conversely, any branching point also determines the value of $a_j-a_i$.
        \end{lemma}
        \vspace{\baselineskip}
    
        We denote the subgroup of the monodromy group generated by the one-parameter families $C_{i,p}$ as the EH group. Using Lemmas \ref{curvegrd} and \ref{branch},
         we can determine how the EH group acts on the Young tableaux.
    
         The distance between two boxes in a Young diagram is defined as the sum of their horizontal and vertical distances. Specifically, the distance between adjacent boxes is $1$. Following
          Edidin's arguments in \cite{Ed91} Proposition 1, we establish that:
     \begin{proposition}
     The EH group is generated by elements $\{\pi_{t,a}: t,a\in\mathbb{Z}, 
     1\leq t<g, a>0\}$.
     The action $\pi_{t,a}$ exchanges the entries $t$ with $t+1$  in all Young
      tableaux of shape $\sigma$,
      where $t$ and $t+1$  are in different rows and columns and 
      have a distance of $a$.
     \end{proposition}
    
     \begin{proof}
      Let $L_1$ and $L_2$ be limit linear series  
      on $C_{\infty}$ and correspond to sequences
      \begin{equation*}
        \Delta(L_1):\: \sigma_{\alpha^{(1)}}=\sigma_\alpha,\:
        \cdots,\sigma_{\alpha^{(t-1)}},\:\sigma_{\alpha^{(t)}},\:
        \sigma_{\alpha^{(t+1)}},\:
        \cdots,:\sigma_{\alpha^{(g+1)}}=\sigma_\beta^{\vee}
        \end{equation*}
        \begin{equation*}
          \Delta(L_2):\: \sigma_{\beta^{(1)}}=\sigma_\alpha,\:
          \cdots,\sigma_{\beta^{(t-1)}},\:\sigma_{\beta^{(t)}},\,
          \sigma_{\beta^{(t+1)}},\:
          \cdots,:\sigma_{\beta^{(g+1)}}=\sigma_\beta^{\vee}
          \end{equation*}

     We translate the conditions on vanishing sequences on the limit $\mathfrak{g}^r_d$ on
      the curve $D'$  in Lemma \ref{curvegrd} into conditions 
      on the ramification sequences $\alpha^{(1)}$ and $\alpha^{(2)}$: 
      there exist $0\leq i<j\leq r$ such that 
        \begin{enumerate}
          \item if $i>0$,
        $\alpha^{(1)}_{i-1}<\alpha^{(1)}_i;$
          \item $\alpha^{(1)}_{j-1}<\alpha^{(1)}_j;$
          \item $\alpha^{(2)}_i=\alpha^{(1)}_i+1;$
          \item $\alpha^{(2)}_j=\alpha^{(1)}_j+1;$
          \item $\text{for any } k\neq i,j, \: \alpha^{(2)}_i=\alpha^{(1)}_i+2.$
        \end{enumerate}


     Therefore, some family $C_{t,p}$ exchanges $L_1$ with $L_2$ only if 
     there exist $1\leq t<g$ and $0\leq i<j\leq r$ such that for 
     any $k\neq i,j$, we have $\alpha^{(t+1)}_k=\beta^{(t+1)}_k$,
     and
     \begin{equation*}
     \alpha^{(1)}=\beta^{(1)},\:\cdots,\:\alpha^{(t)}=\beta^{(t)},\:
     \alpha^{(t+2)}=\beta^{(t+2)},\:\cdots,\:\alpha^{(g+1)}=\beta^{(g+1)}.
     \end{equation*}
     This is because $\alpha^{(t)},\alpha^{(t+2)}$ correspond to
      vanishing sequences
      $a^{(t)},a^{(t+2)}$, which satisfy  the conditions 
     in Lemma \ref{branch}. Thus,
     there exists a monodromy action exchanging the restriction of $L_1$ and $L_2$
      on the union of $Y_t$ and $Y_{t+1}$ in $C_\infty$, leaving the other
       parts 
      of the limit linear series unchanged, thereby exchanging 
      the limit linear series $L_1$ with $L_2$ on 
      $C_\infty$. Let $a=a^{(t-1)}_j-a^{(t-1)}_i$. 
      From Lemma \ref{branch}, we know that the branching points are 
      determined by $a$, indicating that  the number of 
      all possible branching points is finite. 
      Hence, there exists  a monodromy action that 
      exchanges the limit linear series with
       $a^{(t-1)}_j-a^{(t-1)}_i=a$. This action, denoted by $\pi_{t,a}$, is
        the product of all such transpositions and 
        is determined only by  $t$ and $a$.
     
     Let $S$ and $T$  be the Young tableaux representation of 
      $L_1$ and $L_2$
     respectively. 
     According to Theorem \ref{chain}, 
     each step in the chain of Schubert cycles for the limit linear series involves keeping one of the Schubert indices unchanged and incrementing
      the rest by $1$, while maintaining the monotonicity of the indices.
      Since $L_1$ and $L_2$ are exactly 
      the same  except for $\sigma_{\alpha^{t+1}}$,  the only difference between $L_1$ and $L_2$
      is the order in which the fixed indices are chosen in the steps:
     \begin{equation*}
     \sigma_{\alpha^{t}}\supset \sigma_{\alpha^{t+1}}
     \supset\sigma_{\alpha^{t+2}}
     \end{equation*}
     and
     \begin{equation*}
     \sigma_{\beta^{t}}\supset \sigma_{\beta^{t+1}}
     \supset\sigma_{\beta^{t+2}}.
     \end{equation*}
      Therefore, 
      when represented using Young tableaux,  the difference between
       $S$ and $T$ is that $t$ and $t+1$ are exchanged.
     
     Since $S$ and $T$ are both Young tableaux, they satisfy the 
     monotonicity of rows and columns, meaning $t$ and $t+1$ cannot
      be in the same row or column in either $S$ or $T$.
      Moreover, according to the construction of  the Young tableau,
      the horizontal distance between $t$ and $t+1$ in $S$ is exactly
       $\alpha^{(t)}_j-\alpha^{(t)}_i$, and 
       the vertical distance between $t$ and $t+1$ in $S$  is exactly $j-i$.
       Given that 
       \begin{equation*}
       a=a^{(t)}_j-a^{(t)}_i=\alpha^{(t)}_j-\alpha^{(t)}_i+(j-i),
       \end{equation*} we conclude  the
        distance between $t$ and $t+1$ in the Young tableaux $S$ and $T$ is 
        exactly
        $a$.
    
    Therefore, the monodromy action $\pi_{t,a}$ on the
     Young tableaux representation of the limit linear series 
     is to exchange $t$ with $t+1$ in all Young tableaux 
     where $t$ and $t+1$  are located in different rows and columns and 
     have a distance of $a$.
    \end{proof}
    
    \begin{example}
      Let $L$ be the limit linear series in the Example \ref{ex:youngtab}. 
      The 
      Young tableau representation $T$ of $L$ is shown on the left 
      in Figure \ref{youngtab1}. 
    \begin{enumerate}
      \item The action 
      $\pi_{3,2}$ exchanges $3$ with $4$ whenever $3$ and $4$ 
      are at a distance  of $2$. Thus $\pi_{3,2}$  transforms
       the tableau $T$ of $L$ into the tableau shown on the right below.
      . 
    \begin{figure}[H]
        \centering
        \includegraphics[width=0.8\textwidth]{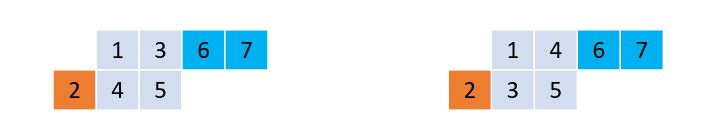}
        \caption{The action of $\pi_{3,2}$}
        \label{youngtab1}
    \end{figure}
       \item The action of $\pi_{3,3}$ leaves $T$ unchanged since $3$ and $4$
       are not at a distance of $3$ in $T$.
       \item For any $k\geq 1$, the action of $\pi_{4,k}$ will fix $T$ because $4$ and $5$ 
       are in the same row in $T$.
   
    \end{enumerate}

  \end{example}

    \section{Doubly Transitivity of the monodromy action}

    Let $YT(g,r,d,\alpha,\beta)$ be the collection of all Young 
    tableaux of shape $\sigma$, then we have
    
    \begin{proposition}\label{transitive}
    The action of the EH group on $YT(g,r,d,\alpha,\beta)$ is transitive.
    \end{proposition}
    
    \begin{proof}
    Assign a lexicographic order to the set
     $YT(g,r,d,\alpha,\beta)$:  
     column-wise from left to right and top to bottom within each column. 
     In other words, the smallest Young tableau $S$ is the one
      where $1,2,\cdots,g$ are sequentially filled into 
     each column from left to right, and within each column from top to bottom.
      We only need to prove that any Young tableau
       can be moved to $S$ by an EH group element. 
       In fact, we only need to prove that for any Young tableau $T\neq S$, 
       there exists an EH group element that move it to a Young 
       tableau $T'$ that is smaller than $T$ in the lexicographic order. Since 
       this process will eventually 
       stop after a finite number of steps, the final tableau must be $S$.

    Given a Young tableau $T\neq S$, let $M$ be the 
    first integer different from $S$ in the lexicographic order. 
    Then $M-1$ must not be to the left or above $M$ in $T$, 
    because the left and upper sides of $M$ in $T$ are the 
    same as those in $S$ and $M$ is the first integer different from $S$. 
    Additionally, according to the definition of Young tableaux, 
    the integers below and to the right of $M$ are larger than $M$, 
    so $M$ and $M-1$ are in different rows and columns. Let $a$ be their 
    distance, then the monodromy action
     $\pi_{M-1,a}$ in the EH group moves $T$ to $T'$, 
     which is smaller than $T$ in the lexicographic order. 
     
    \end{proof}

    \begin{theorem}\label{irred}
    Let $\alpha,\beta$ be two ramification sequences and suppose that
    \begin{equation*}
    \rho(g,r,d,\alpha,\beta)=g-(r+1)(g-d+r)-|\alpha|-|\beta|=0.
    \end{equation*}
    Then there exists a family of twice-marked smooth projective curves 
    $(\mathcal{C}/B,p,q)$ such that $G^r_d(\mathcal{C}/B,(p,\alpha),(q,\beta))$ is
    irreducible.
    \end{theorem}
    Actually, it can be seen from the proof that this theorem is true
    for any sufficiently small irreducible family of 
    twice-marked smooth curves, whose stable limits contain all curves of the  
    form in Figure \ref{reduciblecurve}.
    
    \begin{figure}[H]
        \centering
        \includegraphics[width=0.3\textwidth]{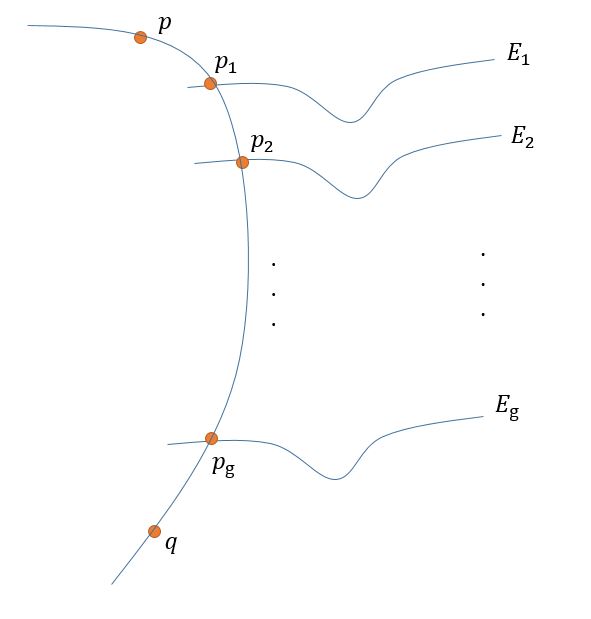}
        \caption{Reducible curve}
        \label{reduciblecurve}
    \end{figure}
    
    \begin{proof}
    Let $(\mathcal{C}/B,p,q)$ be an irreducible family 
    of twice-marked smooth curves and its stable limit contains all stable curves of the form in Figure \ref{reduciblecurve}. Then
    the stable curve obtained by letting $p_i$ 
    approach $p_{i+1}$ will add 
    another rational curve $\mathbb{P}^1$. Repeatedly performing this 
    operation will eventually obtain the curves $C_{i,p}$ and $C_\infty$. 
    Therefore, 
the stable limit of this family include all $C_{i,p}$ and $C_\infty$. 
    
    Shrink $B$ if necessary so that the family 
    $G^r_d(\mathcal{C}/B,(p,\alpha),(q,\beta))$ is smooth. To prove that 
    $G^r_d(\mathcal{C}/B,(p,\alpha),(q,\beta))$ is irreducible, 
    it suffices to show that the monodromy action on the fiber is transitive.
    
    According to \cite{EH86}, the family 
     $G^r_d(\mathcal{C}/B,(p,\alpha),(q,\beta))$ can be extended to
      its stable limit (at least along one-parameter families). 
      Harris \cite{Ha79} proved that
       the monodromy group is birational invariant, 
       so it is enough to show that the action on the fiber
        of $C_\infty$ is transitive. Now by Proposition \ref{transitive}, 
        the group action induced by the one-parameter
         families $C_{i,p}$ on the fiber of $C_\infty$ has
         already been transitive, so we are done.
    \end{proof}
    
    
    More generally, under certain conditions, 
    we can prove that the action of the EH group
     on $YT(g,r,d,\alpha,\beta)$ is doubly transitive.
    
    \begin{definition}
    Let $G$ be a group acting on a set $S$. The action of $G$ on $S$ is 
    \textit{doubly transitive} if for 
    any $x,y,w,z\in S$ with $x\neq y$ and $z\neq w$, there exists 
    a group element $g\in G$ such that $g\cdot x=y$ and $g\cdot z=w$.
    \end{definition}

    \begin{lemma}\label{generaltransitive}
        If the given ramification sequences $\alpha,\beta$ satisfy
        the following conditions:
        \begin{enumerate}
            \item $\max\{\alpha_0-\alpha_{r},1\} < \min_{0\leq j<r}\{\alpha_{r-j-1}+
            \beta_{j}\}+g-d$;
            \item $\sum_{i=0}^{r}(\alpha_i - \alpha_r) < \alpha_r + g - d + r + \beta_r$;
            \item For any $0\leq i < r$, 
            $\alpha_{r-i}+\beta_i  - 1 < 
            \min_{i\leq j<r}\{\alpha_{r-i-j-1}+\beta_{j}\}$;
              \item For any $0 < i <r$, $\alpha_{r-i}+\beta_i =
              \alpha_{r-1}+\beta_1\geq \alpha_r+\beta_0$;  
        \end{enumerate}
        then the action of the EH group 
        on $YT(g,r,d,\alpha,\beta)$ is doubly transitive.
    \end{lemma}
        \begin{proof}
        Similar to the proof of Proposition \ref{transitive}, 
        we assign a 
        lexicographic order to the set $YT(g,r,d,\alpha,\beta)$: 
        column-wise from left to right and 
        top to bottom in each column. 
        Let $S$ be the smallest tableau in this order.
        Let $Z$ be the Young tableau obtained by filling the boxes 
        row-wise from top to bottom and from left to right in each row. 
        By using induction on the size of
         the Young diagram, we can see that $Z$ is the 
         largest tableau in the lexicographic order. 
        
         To establish the doubly transitivity of the EH group action, it is 
         sufficient to prove that any Young tableau can be 
         transformed into $S$ by an EH group element, while 
         preserving the Young tableau $Z$. Actually, it is enough to prove that for
          any Young tableau $T\neq S$, there exists an EH group element $\pi$ 
          that moves $T$ to
          a tableau $T'$ smaller in the lexicographic order, 
          with $\pi$ fixing  the tableau $Z$. Since this process 
          is finite, it will eventually stop, with the final 
           tableau being $S$, while the tableau $Z$ remains unchanged.

Given a Young tableau $T \neq S$, let $M+1$ be the first integer in the 
lexicographic order that differs from $S$. Let $T_{i,j}$  represent the integer 
in the box located at the $i$-th (starting from $1$) row from top to bottom and
 the $j$-th (starting from $1$) column from 
left to right in tableau $T$. The integer $M$
 cannot appear to the left of or above the box containing $M+1$, 
  because at these positions, $T$ and $S$ are identical, and $M+1$ 
  is the first 
  different integer. Therefore, $M$ can only lies above and to the right of
   $M+1$. Let $M+1 = T_{a,b}$, $M= T_{c,d}$.     
         We denote their distance as
        \begin{equation*}
        L=(a-c)+(d-b).
        \end{equation*}
        
       We aim to exchange 
          $M$ with $M+1$ in $T$, while keeping $Z$ fixed, 
          so that $T$ becomes a Young tableau with a smaller lexicographic 
          order. If $M$ and $M+1$ are in the same row in $Z$, 
          then the group action $\pi_{M,L}$ is sufficient to 
          achieve this. 
          Therefore, we only need to consider the case where $M$ and $M+1$ are
           not in the same row. Let's assume  $M$ is at the end of the $i$-th
            row and $M+1$ is at the beginning of the $(i+1)$-th row in $Z$. 
            If the distance between $M$ and $M+1$ in $Z$ is not equal to $L$, 
            i.e.,
    \begin{equation*}
    \beta_{i-1}+\alpha_{r-i}+g-d+r \neq L,
    \end{equation*}
    then the group action $\pi_{M,L}$ can exchange $M$ with $M+1$ in $T$
     while fixing $Z$. Therefore, in the following proof, we always assume
    \begin{equation*}
    \beta_{i-1}+\alpha_{r-i}+g-d+r = L.
    \end{equation*}
    
    Now we consider the position of $M-1$ and let $M-1 = T_{e,f}$. 
    Then the possible 
    positions of $M-1$ are
    \begin{enumerate}
      \item $c<e<a,b<f<d$,
      \item $e<c,f>d$,
      \item $f\leq b$,
      \item $f=d$,
      \item $e=c$,
    \end{enumerate}
    as
     illustrated in the Figure \ref{mposition}.
    
    \begin{figure}[H]
    \centering
    \includegraphics[width=0.75\textwidth]{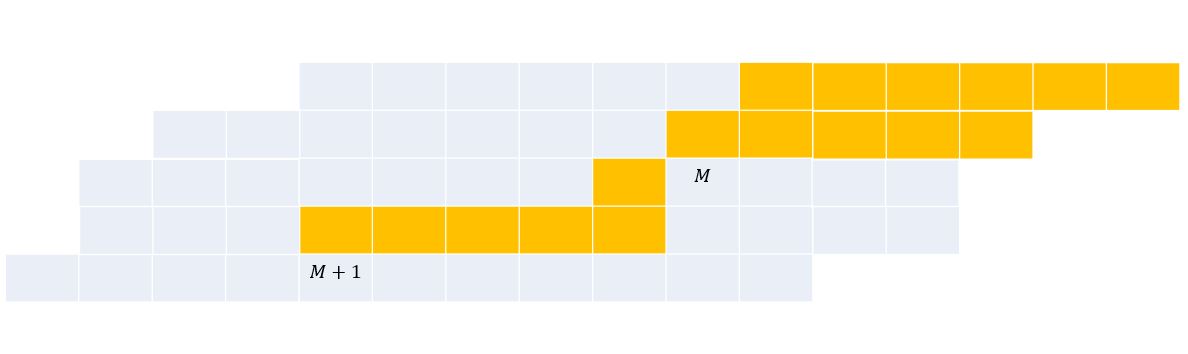}
    \caption{Possible positions of $M-1$}
    \label{mposition}
    \end{figure}
    
    \textbf{Case 1:} If $c<e<a,b<f<d$, we can use the group action 
    $\pi_{M-1,(e-c)+(d-f)}$ to exchange $M$ with $M-1$. Since $M-1$ is in 
    the same row as $M$ in $Z$, $\pi_{M-1,(e-c)+(d-f)}$ fixes $Z$. Then, we apply
     the group action $\pi_{M,(a-e)+(f-b)}$ to exchange $M$ with $M+1$ in $T$. 
    Given that  
    \begin{equation*}
    (a-e)+(f-b)<(a-c) + (d-b)=L,
    \end{equation*}
    the action $\pi_{M,(a-e)+(f-b)}$ fixes $Z$. Thus, the composition of these 
    two group actions exchange $M$ with $M+1$ in $T$ while keeping $Z$ unchanged.
    
    \textbf{Case 2:} If $e<c,f>d$, the proof follows a similar approach to that in Case 1. First we use the group action 
    $\pi_{M-1,(c-e)+(f-d)}$
     to exchange $M$ with $M-1$ while maintaining $Z$. Subsequently, we apply
      the group action $\pi_{M,(a-e)+(f-b)}$
      to exchange $M$ with $M+1$ in $T$. Since
    \begin{equation*}
    (a-e)+(f-b)>(a-c)+(d-b)=L,
    \end{equation*}
    the action $\pi_{M,(a-e)+(f-b)}$ fixes $Z$. Therefore, the composition
     of these two group actions achieves the goal.

    \textbf{Case 3:} If $f=b$, meaning $M-1$ and $M+1$ are in the same column, then $M-1$ 
    must be directly above $M+1$. If not, $M$ would also have
     to be positioned between  $M-1$ and $M+1$, contradicting our assumption.
     In this case, $M$ must be positioned in one of the upper-left corners of 
     the columns to the right of the column containing
     $M+1$. This is because the smallest integer in these columns is $M$. Additionally, 
      according to assumption (1), we have
    \begin{equation*}
    \begin{aligned}
    L &= (a-c)+(d-b)\\
    &\leq \alpha_0-\alpha_{r}+r\\
    &< \beta_{i-1}+\alpha_{r-i}+g-d+r\\
    &=L,
    \end{aligned}
    \end{equation*}
    which is a contradiction. Therefore, this case cannot occur.
    
    If $f<b$, indicating that $M-1$ is in one of the columns to the left
     of $M+1$, then
     $M$ must be in the first box of the column it belongs to. As in the
      previous case, $M$ must be in one of  the 
    upper-left corners of the columns to the right of the one 
    containing $M+1$.  Thus, by using a similar argument, 
    we arrive at a contradiction in this case as well.

    \textbf{Case 4:} If $f=d$, i.e., $M-1$ is above $M$, then 
    it must be adjacent to $M$; otherwise, there would be no integer 
    between $M-1$ and $M$ to fill in. 
    In this case, we can first exchange $M$ with $M+1$ in $T$
                  using $\pi_{M,L}$. Note that this exchange also
                   affects $M$ and $M+1$ in $Z$. 
                   Next, we exchange $M-1$ and $M$ in $\pi_{M,L}(T)$ 
                   using $\pi_{M-1,L+1}$. Since $M-1$ and $M$ in 
                   $\pi_{M,L}(Z)$ 
                   are separated by a distance of $L-1<L+1$, they remain
                    unchanged in $\pi_{M,L}(Z)$. Finally, we exchange 
                    $M$ with 
                     $M+1$ in $\pi_{M-1,L+1}(\pi_{M,L}(Z))$ using $\pi_{M,L}$,
                      resulting in
               \begin{equation*}
               \pi_{M,L}(\pi_{M-1,L+1}(\pi_{M,L}(Z)))=Z,
               \end{equation*}
               while $M+1$ and $M$ in $\pi_{M-1,L+1}(\pi_{M,L}(T))$ are in
                the same column and thus remain unchanged. 
                Therefore, $\pi_{M,L}\circ\pi_{M-1,L+1}\circ\pi_{M,L}$ fixes 
                $Z$ and 
                 reduces
                 $M+1$ to $M-1$ in $T$, hence
               \begin{equation*}
               \pi_{M,L}\circ\pi_{M-1,L+1}\circ\pi_{M,L}(T)<T.
               \end{equation*} 
               This process does not require the conditions.

    \textbf{Case 5:} 
    If $M-1$ is in the same row as $M$, then $M-1$ must be
    adjacent to $M$. Note that we have already addressed Case 3, 
    so we assume $f\neq b$.
In this case, we consider the position of $M-2$. 
    If $M-2$ is also in the same row and adjacent to $M-1$, we continue 
    this process with $M-3$, and so on, until we find $M-k$ such that
     $M-k, M-(k-1), \dots, M-1, M$ are
     consecutive in the same row,  but
     $M-(k+1)$ is
      not. Let $M-k=T_{c,g}$. Since $a>c$, 
      it must be that $g\geq b$.

    \textbf{Subcase 5.1: } If $g=b$, meaning that  $M-k$ is in the same 
    column as $M+1$, 
    then $M-k$ must 
    be adjacent to $M+1$, as the integers $M-(k-1),\cdots,M$ are
     not in this column.
     Therefore, $a=c+1$, which means $M+1$ and $M$ are in adjacent rows. 
     Additionally, since $M+1$ 
     is the first integer in the Young tableau $T$ that differs from $S$,
      there are no integers lying above $M-k,\cdots,M$.
     Hence, the condition (1) implies that $M$ must be in the first row. 
     According to the condition (3)
     \begin{equation*}
      \begin{aligned}
        \alpha_r+\beta_0-1<L,
      \end{aligned}
      \end{equation*}
     in the second row, there must be 
     exactly 
     $\alpha_{r-1} - \alpha_r$ boxes to the left of $M+1$
     boxes. The condition (2) implies that 
    \begin{equation*}
    \begin{aligned}
      M< \alpha_r+\alpha_{r-1}+2*(g-d+r)+\beta_0+\beta_1,
    \end{aligned}
    \end{equation*}
    hence $M$ can only be at the end of the first row in $Z$. This means that the first
    row of $T$ is $1, \cdots, M$, namely, the first row of $T$ is exactly the same as
    the first row of $Z$.

    We proceed by using induction on the number of rows $r+1$ to prove 
        that, when the first row of $T$ is the same as that of $Z$, $T$ can be moved to a lexicographically smaller 
        tableau $T'$ through an EH group action. If $r=1$, then the
         Young diagram has only two rows. The first row of $T$ is identical
          to that of $Z$, so there is only one way to fill the
           second row, which implies $T$ is the same as $Z$. However, 
          this contradicts the assumption that $T\neq Z$. Therefore, this
           case cannot occur when $r=1$. 
           
           If $r=s>1$, 
           by the induction hypothesis, for Young diagrams with $s$ 
           rows satisfying the conditions, the action of the EH 
           group on all of its Young tableaux is doubly transitive. 
           let $\tilde{\sigma}$ 
            be the sub-Young diagram of $\sigma(g,r,d,\alpha,\beta)$ 
             obtained by
            removing the first row. Similarly, let $\tilde{T}$ and $\tilde{Z}$ be
             the Young tableaux obtained by removing the first row from
              $T$ and $Z$ respectively. Then the
               $\tilde{\sigma}$ has $s$ rows and  $\tilde{T}$, $\tilde{Z}$ are 
                Young tableaux 
                of $\tilde{\sigma}$ satisfying the conditions.
          
               We can identify actions $\pi_{t,a}$  
               on the set of Young tableaux of $\tilde{\sigma}$ with
                actions $\pi_{t+M,a}$
               on the tableaux of $\sigma$ that have the same first row 
               as 
               $Z$, since  actions $\pi_{t+M,a}$ do not change the integers
               $1,\cdots,M$ in the 
               first row.
                If $\tilde{T}$ is not the lexicographically smallest
                Young tableau of diagram $\tilde{\sigma}$,
                 then by the induction hypothesis, 
                 there exists an EH group action $\pi_{t,a}$ such that
                  $\pi_{t,a}(\tilde{T})<\tilde{T}$ in the 
                  lexicographic order and
                   consequently 
                  $\pi_{t+M,a}(T)<T$. 
                  If $\tilde{T}$ is the lexicographically smallest Young 
                  tableau of $\tilde{\sigma}$, then since $s>1$,  
                  $M+2$ must be in the same column as $M+1$ and directly below $M+1$.
               Now, we first exchange $M$ with $M+1$ in $T$ using $\pi_{M,L}$, 
               which simultaneously changes $Z$. 
               Next, we   apply $\pi_{M+1,L+1}$ to exchange $M+1$ and $M+2$ 
               in $\pi_{M,L}(T)$. Since the distance between
               $M+2$ and $M+1$ in $\pi_{M,L}(Z)$ is $M-1<M+1$, 
               no changes occur in $\pi_{M,L}(Z)$. 
               Finally, we exchange $M$ with $M+1$ in 
               $\pi_{M+1,L+1}(\pi_{M,L}(Z))$ using $\pi_{M,L}$, 
               resulting in
               \begin{equation*}
               \pi_{M,L}(\pi_{M+1,L+1}(\pi_{M,L}(Z)))=Z.
               \end{equation*}
               Because $M+1$ and $M$ in $\pi_{M+1,L+1}(\pi_{M,L}(T))$ are
                in the same column, they remain unchanged. 
                Thus, $\pi_{M,L}\circ\pi_{M+1,L+1}\circ\pi_{M,L}$ fixes $Z$ and reduces
                 $M+1$ to $M$ in $T$, hence
               \begin{equation*}
               \pi_{M,L}\circ\pi_{M+1,L+1}\circ\pi_{M,L}(T)<T.
               \end{equation*}
               Therefore, the claim is proved according to 
               the induction hypothesis. Also 
               note that this induction depends on proving
                other cases of the lemma, but other cases do not 
               need induction, thus avoiding circular reasoning.
    
    \textbf{Subcase 5.2: } If $g>b$, that is
     $M-k$ and $M+1$ are in different columns, then 
    we consider 
    the position of $M-(k+1)$. 

    \textbf{Subcase 5.2.1: } If $M-(k+1)=0$, then $M-k=1,\cdots, M=k+1$, 
    hence there is no integer above $M$. If $M$
    is not in the first row, then by assumption (1), 
    \begin{equation*}
      \begin{aligned}
      L=(a-c)+(d-b) &\leq \alpha_0-\alpha_r+r\\
      &< \alpha_{r-i+1}+g-d+r+\beta_{i-1}\\
      &= L,
      \end{aligned}
    \end{equation*}
    which is a contradiction. Therefore, $1,\cdots, k+1=M$ is the first row of 
    $T$.
    This means that the first row of $T$ is exactly the same as
          the first row of $Z$. We have already addressed such case in Case 5.1,
           so we are done.

    \textbf{Subcase 5.2.2: } If $M-(k+1)\geq 1$, Let
    $M-(k+1)=T_{u,v}$. Similarly, the possible positions
     for $M-(k+1)$ are 
   \begin{enumerate}
     \item $c<u<a,b<v<g,$
     \item $u<c,v>d,$
     \item $v=d,$
     \item $v=b$.
   \end{enumerate}
   We can sequentially exchange adjacent pairs of $M-(k+1),M-k,\cdots, M-1$  
               to move $M-1$ to the original position
               of $M-(k+1)$ in $T$. This reduces to Cases
                1, 2, or 3, which have been proved earlier. 
                Thus, through a composition of several group actions, 
     we can exchange $M$ with $M+1$ in the tableau $T$ to 
     obtain a Young tableau with a smaller lexicographic order, 
                However, 
                we need to check that $Z$ remains
                 unchanged after these actions. 
    
  If $i=1$, i.e., $M$ is at the end of the first row in $Z$, then 
  since $M-(k+1)>0$, the numbers
  $M-(k+1)$, $M-k,\cdots,M-1$ are 
  also in the first row in $Z$.  Hence the actions leave $Z$ unchanged. 
  If $1<i<r$, the length of the $i$-th row is at 
  least as long as the row containing $M$ in $T$. Consequently,
   $M-(k+1)$, $M-k,\cdots,M$ all lie in the 
  same row of $Z$, keeping $Z$ unchanged. Since $M$ and $M+1$ are in 
  different rows, it follows that  $i\neq r$. 
  Therefore, we have completed the proof of Case 5.

    Combining all the cases above, we are done.
         
        \end{proof}
    
    \begin{example}
      \begin{enumerate}
        \item If $r=1$, then the Young diagram $\sigma(g,r,d,\alpha,\beta)$
        consists of $2$ rows. In this case, all the conditions in Lemma \ref{generaltransitive}
        become vacuous. Therefore, for any ramification sequences $\alpha,\beta$,
         the EH group is doubly transitive.
        \item If $\alpha=\beta=(0,0,\cdots,0)$, which corresponds to  the unramified case, then 
        the conditions of Lemma \ref{generaltransitive} simplify to 
        $r+1<g-d+r$ , so this lemma serves as 
        a generalization of Edidin's theorem from \cite{Ed91}.
        \item This lemma also leads to several interesting cases. For instance, if $\alpha
        =\beta=(1,\cdots,1,0)$, which are known as \textit{cusps}, 
        then the conditions of Lemma 
        \ref{generaltransitive} simplify to $r < g -d+r$, so
         the EH group acts doubly transitively
        in this case.

        \begin{figure}[H]
          \centering
          \includegraphics[width=0.4\textwidth]{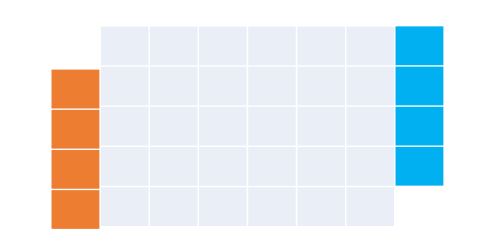}
          \label{cusps}
        \end{figure}

        \item If $\alpha=(\alpha_0,\alpha_1,\alpha_2=0)$, $\beta=(0,0,0)$, then 
        the conditions of Lemma \ref{generaltransitive} simplify to the following: 
        \begin{enumerate}
          \item $\max\{\alpha_0-\alpha_1,1\} < \alpha_1 + g - d$,
          \item $\alpha_0+\alpha_1<g-d+2$.
        \end{enumerate} 
        For instance, when $(g,r,d)=(16,2,14)$ and $\alpha=(3,1,0)$, 
        the EH group acts doubly transitively.
        \begin{figure}[H]
          \centering
          \includegraphics[width=0.4\textwidth]{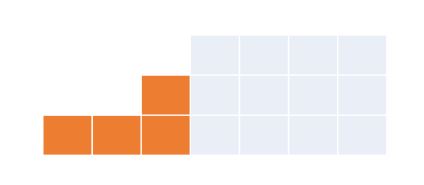}
          \label{ramf}
        \end{figure}
      \end{enumerate}
    \end{example}

        \begin{remark}
        The section 3 of \cite{Ed91}  addressed the unramified case, 
        but the proof has some gaps
        : the case (4a) in \cite{Ed91} does not consider when $T\in Y(m,n)$ is the smallest 
        element, and the case (3) does not consider when 
         $M-c$ is above $M$. Here, we provide a complete proof and
          simplify the arguments for case (5) in \cite{Ed91} .
        \end{remark}    
        
        Using the doubly transitivity of the EH group action, 
        we can establish that
        the EH group is actually a very
         large subgroup. Under certain conditions, 
         it is either the alternating group or the full symmetric group. 
         To begin, let us revisit a classical theorem in group theory.
    
        \begin{theorem}[Bochert]
        Let $G$ be a doubly transitive subgroup of the symmetric 
        group $S_n$. If there exists an element $g \neq 1$ in $G$ such that
         the number of elements moved by $g$ is fewer than $\frac{n}{3}-2\frac{\sqrt{n}}{3}$, then $G$ is either the symmetric group $S_n$ or the alternating group $A_n$.
        \end{theorem}
        
        This theorem  tells us that we need to find 
        a nontrivial element that moves
         only a sufficiently small number of elements.
        \begin{lemma}\label{movefew}
        If $r \geq 2$ and $\alpha_{r}+\beta_{r}+g-d+r > r+1$, then 
        there is an element of the EH group that 
         moves at most $\frac{1}{4}$ of the Young tableaux.
        \end{lemma}
        \begin{proof}
        Let $s$ be the length of the first column of $\sigma$ and
        $t=\alpha_{r}+\beta_0+g-d+r$ be the length of the first row. 
        Additionally,  let $L=\alpha_0+\beta_0+g-d+2r-1$ be the distance
         from the lower-left corner to the upper-right corner of the 
         Young diagram $\sigma=\sigma(g,r,d,\alpha,\beta)$.
        
        If the ramification sequence $\alpha$ satisfies
        \begin{equation*}
        \alpha_0=\alpha_1=\cdots=\alpha_{r},
        \end{equation*}
        we consider the EH group element $\pi_{s+t-2,L}$. 
       Let $A$ be the set of all Young tableaux  moved 
       by $\pi_{s+t-2,L}$, i.e., those not 
       fixed by $\pi_{s+t-2,L}$.
        Let $T$ be a Young tableau in $A$. 
        Since $L$ is the farthest distance
         between two boxes in $\sigma$ and 
        \begin{equation*}
        \alpha_{r}+\beta_{r}+g-d+r> r+1,
        \end{equation*}
        $L$ can only be the distance between the
         lower-left corner $T_{s-1,0}$ and the upper-right corner $T_{0,t-1}$
         of $\sigma$. 
         Thus, the integers in the lower-left box or
          the upper-right box  must be $s+t-2,\, s+t-1$. 
          Since the total number of boxes in the first column and row 
          is $s+t-1$ and the integers filled in the boxes except 
          the lower-left and upper-right boxes must be less than $s+t-2$,
           the numbers filled  in the first row and column 
           are exactly $1,2,\cdots,s+t-1$. 
           Now, consider $s+t$ and $s+t+1$. Since $s+t$ is the smallest 
           integer after removing the first row and first column, 
           we have $T_{1,1}=s+t$. 
           Similarly, $s+t+1$ is the second smallest
            integer after removing the first row and first column, 
            so either $T_{1,2}=s+t+1$ or $T_{2,1}=s+t+1$.
        
        We construct three Young tableaux $T_1,T_2,T_3$ from $T$. The tableau $T_1$ 
        is obtained by exchanging  $s+t-1$ with $s+t$, $T_2$ by
         exchanging $s+t-2$ with $s+t$, and $T_3$ by first exchanging 
          $s+t-1$ with $s+t$, and then $s+t-2$ with $s+t+1$. 
          Due to the  special positions
           of $s+t$ and $s+t+1$, the resulting 
         tableaux are still Young tableaux. Since each of 
         $T_1,T_2,T_3$ 
         has at least one of  $s+t-1$ and $s+t-2$ not in the lower-left corner or 
         the upper-right corner,
         they will all be fixed by $\pi_{s+t-2,L}$. 
         
         For any $S\in A$ with $S\neq T$,
          we can similarly construct the corresponding three tableaux $S_1,S_2,S_3$.
          Next, we show 
          that $S_1,S_2,S_3$ and $T_1,T_2,T_3$ are  distinct Young tableaux. 
          By construction, $S_1$ is different from $S_2$ and $S_3$. Moreover, 
          since the lower-left corner and the upper-right corner 
           of $S_1$ are $s+t,s+t-2$, and those of $T_2,T_3$ are $s+t,s+t-1$ and
           $s+t,s+t+2$ respectively, $S_1$ is different from $T_2,T_3$.
            The tableaux $S_1$ and $T_1$ are obtained by the actions 
            $\pi_{s+t-1,a}$ and $\pi_{s+t-1,a'}$ on $T$ and $S$, respectively.
            If $a\neq a'$, then the distances between $s+t-1$ and $s+t$ are 
            in $T_1$ and $S_1$ are different, so $T_1\neq S_1$. If $a=a'$,
            then $\pi_{s+t-1,a}=\pi_{s+t-1,a'}$. Since $T\neq S$, we have
            \begin{equation*}
              S_1 = \pi_{s + t - 1, a'}(S) \neq \pi_{s + t - 1, a}(T) = T_1.
              \end{equation*}  
              Therefore, $S_1$ is different from $S_2,S_3,T_1,T_2,T_3$.
             Similarly, by symmetry, $S_2$ is distinct from $S_1,S_3,T_1,T_2,T_3$ and 
             $S_3$ is distinct from $S_1,S_2,T_1,T_2$ as well.
              $S_3$ and $T_3$ 
             are obtained from $S_1$ and $T_1$ respectively 
             by exchanging $s+t-2$ with $s+t+1$. If the positions of 
              $s+t-2,s+t+1$ 
              $S_3$ and $T_3$ are different, then 
              $S_3\neq T_3$. If the positions of $s+t-2,s+t+1$ 
              in $S_3$ and $T_3$ are the same, then  the positions of
               $s+t-2$ and $s+t+1$ would be
               completely identical, contradicting the distinctness
                of $S_1$ and $T_1$. 
                Hence, $S_3$ and $T_3$ must be distinct. 
                Therefore, $S_1,S_2,S_3,T_1,T_2,T_3$ are mutually
                 distinct Young tableaux.
        
        By constructing $T_1$ for each $T$ in $A$, we obtain a new set $A_1$
        of Young tableaux. Similarly, sets $A_2,A_3$ can be obtained. 
        As argued earlier, the sets $A,A_1,A_2,A_3$ have equal cardinality and 
        are pairwise disjoint. Therefore, the number of Young tableaux 
         moved by $\pi_{s+t-2,L}$ is at most $\frac{1}{4}$ of the total number of
          Young tableaux.
        
        If the ramification sequence $\alpha$ does not satisfy
        \begin{equation*}
        \alpha_0=\alpha_1=\cdots=\alpha_{r},
        \end{equation*}
        we then consider the EH group element $\pi_{s+t-1,L}$. 
        The proof follows a similar approach to the previous case.
         Let  $A$ be the set of
        all Young
         tableaux moved by $\pi_{s+t-1,L}$. For any $T\in A$,
          the 
         lower-left corner and the upper-right corner of $T$ are $s+t-1$ 
         and $s+t$. We denote the Young diagram obtained from removing the 
         first row and first column of $\sigma$ as $\tilde{\sigma}$.
         If $\tilde{\sigma}$ has only one 
        upper-left corner, it must be  in the first row
         of $\tilde{\sigma}$.  Given that the smallest integer in 
         $\tilde{\sigma}$ is $s+t+1$, the 
        upper-left corner of $\tilde{\sigma}$ must be $s+t+1$, and $s+t+2$
         must be adjacent to $s+t+1$ either to the right or below.
          Since the 
        upper-left corner is in the first row of $\tilde{\sigma}$, 
        $s+t+2$ cannot be in the same row as $s+t-1$ or $s+t$.
        
        For $T$ we construct $T_1,T_2,T_3$ as follows: 
        $T_1$ is obtained from $T$ by exchanging $s+t-1$ with $s+t+1$,
         $T_2$ by exchanging $s+t$ with $s+t+1$, and $T_3$ 
         by first exchanging $s+t$ with $s+t+1$ and then $s+t-1$ with $s+t+2$.
          The resulting tableaux are still Young tableaux. 
          By arguments analogous to those used previously, 
          $S_1,S_2,S_3$ and $T_1,T_2,T_3$ are pairwise distinct Young tableaux.
           Therefore, similarly, we can conclude that the  number of Young tableaux moved by
            $\pi_{s+t-1,L}$ is at
            most $\frac{1}{4}$ of all Young tableaux.
            
          If $\tilde{\sigma}$ has more than one upper-left corner, 
          we choose any two of them, denoted as $T_{i,j}$ and $T_{p,q}$.
           For any $T \in A$, 
            we construct corresponding tableaux $T_1, T_2, T_3$ as follows: 
            $T_1$ is obtained from exchanging $T_{s-1,0}$ with $T_{i,j}$,
            $T_2$ by exchanging $T_{s-1,0}$ with $T_{p,q}$,
            $T_3$ by exchanging $T_{0,t-1}$ with $T_{i,j}$.
            Since both $T_{i,j}$ and $T_{p,q}$
  are upper-left corners,
   the resulting tableaux remain Young tableaux after these exchanges.
   Moreover, the tableaux $T_1,T_2,T_3$ are fixed by the EH group element
   $\pi_{s+t-1,L}$. Finally, using similar reasoning, 
   we conclude that the number of Young tableaux moved by $\pi_{s+t-1,L}$
   is at most $\frac{1}{4}$ of all Young tableaux.

        \end{proof}

        \begin{theorem}\label{alter}
            Given two ramification sequences $\alpha,\beta$. Suppose
            \begin{equation*}
            \rho(g,r,d,\alpha,\beta)=g-(r+1)(g-d+r)-|\alpha|-|\beta|=0
            \end{equation*}
            and $r\geq 2$, $\alpha_{r}+\beta_r+g-d+r>r+1$. Let $N=N(g,r,d,\alpha,\beta)$
            denote the number of Young tableaux of the Young diagram $\sigma=\sigma(g,r,d,\alpha,\beta)$.
            If $\alpha,\beta$ satisfy the conditions of 
            Lemma \ref{generaltransitive},
            then the EH group is either the alternating group $A_N$ or the symmetric group $S_N$.
            \end{theorem}
            \begin{proof}
            If the ramification sequences $\alpha,\beta$ satisfy the conditions
             of Lemma \ref{generaltransitive},
            then the action of the EH group on the set of all Young tableaux $YT(g,r,d,\alpha,\beta)$ is 
            doubly transitive.
            Next, we demonstrate that the EH group also satisfies the 
            condition of Bochert's theorem.

            The Young diagram $\sigma$ contains
            a maximal rectangle-shaped sub-Young diagram $\tau$ of
             size $(r+1)\times(\alpha_{r}+\beta_{r}+g-d+r)$. Its left
             and right sides are also  sub-skew Young diagrams, which we denote 
             by $\sigma_1$ and $\sigma_2$, respectively.
            Since $r\geq 2$,
            $\sigma$ has at least $3$ rows. Furthermore, 
            \begin{equation*}
            \alpha_{r}+\beta_{r}+g-d+r>r+1,
            \end{equation*}
            ensuring that 
             the size of the rectangle-shaped Young diagram $\tau$
            is at least $3\times 4$. 
            Let $m$ be the number of boxes in $\sigma_1$ and 
            $n$ be the number of boxes in $\sigma_2$.
            We can fill $1,2,\cdots,m$ lexicographically
             into $\sigma_1$ and $g-n+1,\cdots,g$ into $\sigma_2$.
            By filling $m+1,\cdots,g-n$ into the middle
             rectangle $\tau$ such that  $\tau$ forms a Young tableau,
            this filling yields a Young tableau of 
            the entire Young diagram $\sigma$. Conversely, if $m+1,\cdots,g-n$
            are filled in such a way that $\sigma$ forms a Young tableau, 
            then $\tau$  becomes a Young tableau as well.
            Hence the number of Young tableaux of the (skew)
             Young diagram $\sigma$ is at least the number of Young tableaux
            of the shape $(r+1)\times(\alpha_{r}+\beta_{r}+g-d+r)$.
            On the other hand, by considering placing a $3\times 4$ subrectangle 
            at the upper-left  of $\tau$, we have
            \begin{equation*}
            N(g,r,d,\alpha,\beta)\geq N(r+1,\alpha_{r}+\beta_{r}+g-d+r)\geq N(3,4).
            \end{equation*}
            For Young diagrams, the hook-length formula calculates the number 
            of all Young tableaux. Specifically,
            for a rectangle-shaped diagram of size $s\times t$, according to the hook-length formula,
            the number of all Young tableaux is
            \begin{equation*}
            N(s,t)=(st)!\prod_{i=0}^{s-1}\frac{i!}{(t+i)!}.
            \end{equation*}
            Thus, we have 
            \begin{equation*}
            N=N(g,r,d,\alpha,\beta)\geq N(3,4)=(12)!\prod_{i=0}^{2}\frac{i!}{(4+i)!}=231.
            \end{equation*}
           This implies
            \begin{equation*}
            (\frac{N}{3}-2\frac{\sqrt{N}}{3})\frac{1}{N}=\frac{1}{3}-\frac{2}{3\sqrt{N}}>\frac{1}{4},
            \end{equation*}
            Therefore, to apply Bochert's theorem, we only need to  identify 
            a nontrivial element in the EH group that moves at most 
            $\frac{1}{4}$
            of the Young tableaux. Lemma \ref{movefew} constructs
              such an element, establishing that the EH group is at least the
               alternating group. Since the EH group is a subgroup of 
               the monodromy group, it follows that the monodromy group is also 
               at least the alternating group, hence we are done.
            \end{proof}
            \vspace{\baselineskip}
            
            \section{Monodromy groups of families of linear series 
            with imposed ramifications}

            In this section, we utilize the lemmas proved 
            earlier to establish the main results.
            
            \begin{theorem}\label{monodromyalter}
            Given two ramification sequences $\alpha,\beta$. Suppose
            \begin{equation*}
            \rho(g,r,d,\alpha,\beta)=g-(r+1)(g-d+r)-|\alpha|-|\beta|=0
            \end{equation*}
            and $r\geq 2$, $\alpha_{r}+\beta_r+g-d+r>r+1$. Let $N=N(g,r,d,\alpha,\beta)$
            denote the number of Young tableaux of the Young diagram $\sigma=\sigma(g,r,d,\alpha,\beta)$.
            If $\alpha$, $\beta$ satisfy the conditions of
             Lemma \ref{generaltransitive},
            then the monodromy group of 
            $G^r_d(\mathcal{C}/B,(p,\alpha),(q,\beta))\to B$ in
             Theorem \ref{irred} is either the alternating group $A_N$ or 
             the symmetric group $S_N$.
            \end{theorem}
            \begin{proof}
            By the proof of Theorem \ref{irred}, the monodromy group of 
            $G^r_d(\mathcal{C}/B,(p,\alpha),(q,\beta))\to B$ is the same as
             the monodromy group of the fiber of the stable 
             limit $C_{\infty}$,
            and the EH group induced by the one-parameter families
             $C_{i,p}$ is a subgroup of the monodromy group.
             Since $\alpha,\beta$
            satisfy the conditions of Theorem \ref{alter}, 
            the EH group is either the alternating group or the
             symmetric group. Consequently,  the monodromy group
            is either the alternating group $A_N$ or the
             symmetric group $S_N$.
            \end{proof}
            \vspace{\baselineskip}

            For $r=2$, we can determine the EH group in some cases.
            \begin{proposition}
              Assume $\rho(g,r,d,\alpha,\beta)=0$. If $\alpha_0=\alpha_1=\alpha_2$, 
              $\beta_0=\beta_1=\beta_2$ and  $\alpha_0+\beta_0+g-d+r=2^i$
               for some integer $i>1$, then 
              the monodromy group of $G^r_d(\mathcal{C}/B,(p,\alpha),(q,\beta))\to B$ in
             Theorem \ref{irred} is the symmetric group.
            \end{proposition}
            \begin{proof}
              The given $\alpha$ and $\beta$ satisfy the conditions of Lemma
              \ref{generaltransitive}, hence by Theorem \ref{monodromyalter}, 
              the monodromy group is at least the alternating group. To show
              that it is indeed the full symmetric group, it suffices
              to find an odd permutation in it.
              Let $L=\alpha_0+\beta_0+g-d+r$. We prove that the EH group
              action 
              $\pi=\pi_{2+L-1,2+L}$ exchanging $2+L-1$ 
              with $2+L$ is odd when $L=2^i$ for some integer $i$. We consider the
               number $n$
              of pairs of Young tableaux exchanged by this action. 
              Given that the set of integers in the first row
              and the first column is determined, there are $L-1$ choices for 
              the first row and column.
              After removing the first row and the first
                column
              the resulting diagram is a $2\times(L-1)$ rectangle. The number of
               Young tableaux of
              this diagram is the Catalan number $C_{L-1}$. A result about the 
              oddity of Catalan numbers (see \cite{AK73}) states that $C_{L-1}$ is odd 
              if and only if $L-1=2^i-1$
              for some integer $i$. Therefore, when $L=2^i$ for some $i$, the number 
              $n=(L-1)*C_{L-1}$ is odd. This implies that $\pi$ is an
               odd permutation, 
              so that the monodromy group is the full symmetric group.
            \end{proof}
            \vspace{\baselineskip}
            
            When $r=1$, we can prove a stronger result:
             the monodromy group is
             indeed the full symmetric group.
            
            \begin{lemma}
            If $r=1$,
            then the EH group is the full symmetric group.
            \end{lemma}
            \begin{proof}
            Let $\sigma$ be the Young diagram associated to the 
            ramification sequences $\alpha,\beta$.
            We consider the EH group element $\pi_{s+t-2,L}$ or $\pi_{s+t-1,L}$ 
            constructed in the proof of Lemma \ref{movefew}, where
            $s$ is the length of the first column of $\sigma$, $t$ is the
             length of the first row of $\sigma$, and $L$ denotes the distance from
              the lower-left corner to the upper-right corner of $\sigma$.
            
            If the first row and the first column share common boxes, we take $\pi=\pi_{s+t-2,L}$.
            Let $T$ be a Young tableau of $\sigma$ and be moved by $\pi$, 
            then the integer  in the lower-left corner of $T$
            is either $s+t-2$ or $s+t-1$ and $T_{0,0}$ must be $1$. Hence, 
            integers in the boxes 
            except for the last one in the first row and the first column are 
            uniquely determined.
            Since $r=1$, $\sigma$ has only two rows. Therefore, integers
             greater than $s+t-1$ can only be arranged sequentially in the 
             remaining boxes of the second row,
             Thus,  the integers in the boxes in $T$, except 
             for the lower-left corner and the upper-right corner, 
             are uniquely determined.
            Therefore, the action of $\pi$ on all Young tableaux is just a 
            transposition, moving only $T$ and $\pi(T)$.
            
            If the first row and the first column do not share common boxes, 
            we take $\pi=\pi_{s+t-1,L}$. Let $T$ be a Young tableau 
            of $\sigma$  moved by $\pi$.
            Similarly, except for the lower-left corner and the upper-right corner, 
            the integers in the boxes of $T$ are uniquely determined.
            Hence the action of $\pi$ on all Young tableaux is also a transposition.
            
            Given that $\alpha,\beta$ satisfy the conditions of Lemma
             \ref{generaltransitive},
            according to Lemmas \ref{generaltransitive}, 
            we know that the EH group is doubly transitive.
            Let $\pi=(T,T')$ be a transposition where $T,T'\in YT(g,r,d,\alpha,\beta)$.
            For any pair of distinct Young tableaux $S,S'$, 
            the  doubly
             transitivity of the EH group ensures there exists an EH group element $g$ such that
            \begin{equation*}
            g(T)=S,\: g(T')=S'.
            \end{equation*}
            Consider the group element $g\circ \pi\circ g^{-1}$. Then
            \begin{equation*}
            g\circ \pi\circ g^{-1}(S)=g\circ\pi(T)= S',\:
            g\circ \pi\circ g^{-1}(S')=g\circ\pi(T')= S.
            \end{equation*}
            For any $U\neq S,S'$, we have
            \begin{equation*}
            g\circ \pi\circ g^{-1}(U)=g\circ \pi(g^{-1}(U))=g(g^{-1}(U))=U,
            \end{equation*}
            showing that $g\circ \pi\circ g^{-1}$ is the transposition $(S,S')$.
            Therefore, the EH group
            includes all transpositions. It follows that the EH group is 
             the full symmetric group.
            
            \end{proof}
            \vspace{\baselineskip}
            
            Thus, we obtain
            \begin{proposition}\label{rank1case}
            Given two ramification sequences $\alpha,\beta$. Suppose
            \begin{equation*}
            \rho(g,r,d,\alpha,\beta)=g-(r+1)(g-d+r)-|\alpha|-|\beta|=0
            \end{equation*}
            and let $N=N(g,r,d,\alpha,\beta)$.
            If $r=1$, then
            the monodromy group of the 
            $G^r_d(\mathcal{C}/B,(p,\alpha),(q,\beta))\to B$ in
             Theorem \ref{irred} is the symmetric group $S_N$.
            \end{proposition}
            \begin{proof}
            Similar to the proof of Theorem \ref{monodromyalter}, the monodromy group 
            of $G^r_d(\mathcal{C}/B,(p,\alpha),(q,\beta))\to B$ contains the EH group.
            Given that the EH group is the full symmetric group,
              it follows that the monodromy group is $S_N$.
            \end{proof}

\bibliographystyle{plain}
\bibliography{mypaper}

\begin{thebibliography}{10}

\bibitem{AK73}
Ronald Alter and Ken~K Kubota.
\newblock Prime and prime power divisibility of catalan numbers.
\newblock {\em Journal of Combinatorial Theory, Series A}, 15(3):243--256,
  1973.

\bibitem{BP87}
Ronald~D Bercov and Robert~A Proctor.
\newblock Solution of a combinatorially formulated monodromy problem of
  eisenbud and harris.
\newblock In {\em Annales scientifiques de l'{\'E}cole Normale Sup{\'e}rieure},
  volume~20, pages 241--250, 1987.

\bibitem{CAPB15}
Melody Chan, Alberto~López Martín, Nathan Pflueger, and Montserrat Teixidor~I
  Bigas.
\newblock Genera of brill-noether curves and staircase paths in young tableaux.
\newblock {\em Transactions of the American Mathematical Society}, 370(5),
  2015.

\bibitem{Ed91}
Dan Edidin.
\newblock The monodromy of certain families of linear series is at least the
  alternating group.
\newblock {\em Proceedings of the American Mathematical Society},
  113(4):911--922, 1991.

\bibitem{EH83}
David Eisenbud and Joe Harris.
\newblock Divisors on general curves and cuspidal rational curves.
\newblock {\em Inventiones mathematicae}, 74(3):371--418, 1983.

\bibitem{EH86}
David Eisenbud and Joe Harris.
\newblock Limit linear series: basic theory.
\newblock {\em Inventiones mathematicae}, 85:337--371, 1986.

\bibitem{EH87}
David Eisenbud and Joe Harris.
\newblock Irreducibility and monodromy of some families of linear series.
\newblock In {\em Annales scientifiques de l'{\'E}cole Normale Sup{\'e}rieure},
  volume~20, pages 65--87, 1987.

\bibitem{Gi82}
David Gieseker.
\newblock Stable curves and special divisors: Petri's conjecture.
\newblock {\em Inventiones mathematicae}, 66(2):251--275, 1982.

\bibitem{Ha79}
Joe and Harris.
\newblock Galois groups of enumerative problems.
\newblock {\em Duke Mathematical Journal}, 46, 1979.

\bibitem{GH80}
Joseph~Harris Phillip~Griffiths.
\newblock On the variety of special linear systems on a general algebraic
  curve.
\newblock {\em Duke Mathematical Journal}, 47(1):233--272, 1980.

\bibitem{Bi23}
Montserrat Teixidor-i Bigas.
\newblock Brill-noether loci with ramification at two points.
\newblock {\em Annali di Matematica Pura ed Applicata (1923-)},
  202(3):1217--1232, 2023.

\end{thebibliography}

\end{document}